\documentclass[a4paper,10pt]{article}

\usepackage[utf8]{inputenc}
\usepackage[english]{babel}
\usepackage[T1]{fontenc}
\usepackage{mathtools} 
\usepackage{stmaryrd} 
\usepackage{amssymb, amsmath, amsbsy, amsthm}
\usepackage{graphicx}
\usepackage{amsfonts}
\usepackage{color}
\usepackage{dsfont}
\usepackage{authblk}
\usepackage[titletoc]{appendix}

\usepackage[colorlinks=true, allcolors=blue, breaklinks=true]{hyperref}
\usepackage{verbatim}
\usepackage{tikz}
\usetikzlibrary{patterns}
\usetikzlibrary{intersections,quotes}
\usetikzlibrary{calc,fadings,decorations.pathreplacing,arrows,positioning,shapes, arrows.meta}
\usepackage{pgfplots}
\pgfplotsset{compat=newest}
\definecolor{vertfonce}{rgb}{0,0.392,0}
\usepackage{geometry}
\graphicspath{{Images/}}

\newenvironment{itemize*}%
{\begin{itemize}%
		\setlength{\itemsep}{0pt}%
		\setlength{\parskip}{0pt}}%
	{\end{itemize}}

\numberwithin{equation}{section}

\newcommand{\defeq}{\overset{\text{\tiny def}}{=}}
\usepackage[normalem]{ulem}
\usepackage{dsfont} 
\newcommand{\seqint}{{0\leq i \leq N-1}}

\newcommand{\dt}{\frac{d}{dt}}
\usepackage{credits}
\definecolor{OliveGreen}{rgb}{0,.5,0}

\theoremstyle{plain}
\newtheorem{thrm}{\thrmname}[section]
\newtheorem{lmm}[thrm]{\lmname}

\newtheorem{prpstn}[thrm]{\prpstnname}
\newtheorem{Hyp}[thrm]{\hpthsmname}
\providecommand{\thrmname}{Theorem}
\providecommand{\lmname}{Lemma}
\providecommand{\crllrname}{Corollary}
\providecommand{\prpstnname}{Proposition}
\theoremstyle{definition}
\newtheorem{dfntn}[thrm]{\dfntnname}
\newtheorem{rmrk}[thrm]{\rmrkname}
\providecommand{\dfntnname}{Definition}
\providecommand{\rmrkname}{Remark}
\providecommand{\hpthsmname}{Hypothesis}

\credit{FD}  {Conceptualization, Data curation, Funding acquisition, Writing -- original draft, Writing -- review \& editing}
\credit{AF}  {Conceptualization, Data curation, Software, Writing -- review \& editing}
\credit{IK}  {Conceptualization, Data curation, Writing -- original draft, Writing -- review \& editing}
\credit{AM}  {Conceptualization, Data curation, Formal analysis, Methodology, Software, Writing -- original draft, Writing -- review \& editing}
\credit{MM}  {Conceptualization, Formal analysis, Funding acquisition, Methodology, Software, Writing -- original draft, Writing -- review \& editing}
\credit{LT}  {Conceptualization, Funding acquisition, Methodology, Writing -- original draft, Writing -- review \& editing}

\begin{document}
	
\title{A Kuramoto phase model to explore the synchronisation of a network of circadian clocks.}

\author[1]{Franck Delaunay}
\author[1]{Antoine Fortune}
\author[1]{Ines Krawczyk}
\author[2]{Anastasia Maréchal}
\author[2,3]{Mathieu Mezache}
\author[2]{Laurent Tournier}

\affil[1]{Université Côte d'Azur, Inserm, CNRS, Institut de Biologie Valrose, France}
\affil[2]{Université Paris-Saclay, INRAE, MaIAGE (UR 1404), 78350 Jouy-en-Josas, France}
\affil[3]{Universit\'e Paris-Saclay, INRIA, MUSCA, 91120 Palaiseau, France}%

\date{\today}




\maketitle

\begin{abstract} 
	We propose a model of the circadian clock in a population of cells based on a network of oscillators, derived from the Kuramoto model.
The coupling between oscillators is described by a global interaction term but we introduce a phase-dependent coupling mechanism, such that oscillators interact only within a specific interval of the cycle corresponding to a specific stage of the circadian cycle.
  We analytically demonstrate that this modified system achieves complete asymptotic phase synchronisation, provided specific conditions on the initial phase distribution and the coupling window length are met. 
  To bridge this theoretical framework with experimental observations, we introduce a signal processing procedure based on wavelet decomposition to extract quantitative oscillatory features from {Per2}::luciferase reporter traces.
 We then calibrate the model against datasets from both wild type and {Cry2} knockout hepatocyte spheroids using a two-step quasi-Monte Carlo filtering algorithm. 
 The comparative analysis reveals significant phenotypic divergence, showing that the Cry2KO condition alters the identifiability landscape of the model's parameters and introduces new compensatory mechanisms that confound the initial population heterogeneity with long-term macroscopic signal decay.
\end{abstract}

\section{Introduction}
	
The circadian clock is a mechanism common to many organisms, from prokaryotes like cyanobacteria to  complex life forms including plants, insects, and vertebrates \cite{dunlap1999molecular}. 
	This mechanism drives an organism's adaptation to environmental changes resulting from the daily light/dark cycle.
	These periodic changes translate into circadian oscillations in most physiological and behavioural processes, with a period of approximately 24 hours.
	In mammals, these oscillations can be observed throughout the body at the organ, tissue, cellular, down to the molecular level \cite{takahashi2017transcriptional}. 
	For a proper physiological function, these clocks, operating at different levels, must maintain some degree of synchronisation.
	Poor coordination between them may alters physiological functions and is generally associated with metabolic disorders, inflammation, or malignant diseases \cite{bass2010circadian}. 
	Conversely, a correct synchronisation between clocks may ensure or restore global homeostasis.
	This inter-clock coordination is achieved in a hierarchical manner \cite{bass2010circadian} with a central clock, located in the suprachiasmatic nuclei (SCN) in the hypothalamus and composed of a small number of neurons oscillating in a robust manner \cite{welsh2010suprachiasmatic} which respond to environmental cues (zeitgebers) such as light and propagates a robust oscillating signal via neuronal and hormonal pathways to peripheral clocks within and outside the brain \cite{balsalobre2000multiple}.
	
	At the bottom of this hierarchical system, each cell exhibits a circadian rhythm, driven by molecular oscillations. 
	This cellular autonomous mechanism involves a complex gene regulatory network including multiple transcriptional and posttranslational feedback loops that interact to generate approximately 24h molecular oscillations \cite{takahashi2017transcriptional}. 
	These oscillations are generally described in terms of circadian time (CT), a normalised time unit that sets the phase of the circadian cycle independently of the time of day. 
	In particular two central regulators, complexes CLOCK:BMAL1 and PER:CRY oscillate in phase opposition, CLOCK:BMAL1 peaking in the middle of the subjective day (around CT8-CT10) in mouse peripheral organs, PER:CRY peaking in the middle of the subjective night (around CT18-CT20). 
	This phase opposition is a hallmark of the cellular clock system in mammals.
	More precisely, in \cite{takahashi2017transcriptional}, one can see that a complete cycle (CT0-CT24) can be decomposed into a succession of several qualitative stages (derepression, activation, transcription, repression and poised state).
	Several models of the regulatory network underlying the interactions between CLOCK:BMAL1, PER:CRY and other key regulators have been proposed, based on ordinary differential equations (ODE) \cite{goodwin1965oscillatory, almeida2020transcription} or Boolean networks \cite{diop2020qualitative}. 
	In \cite{diop2020qualitative, almeida2020transcription}, both the complexes CLOCK:BMAL1-PER:CRY phase opposition and the succession of qualitative stages have been captured.
	Those models remain at the scale of a single cell though, little is known at the scale of a cell population.
	
	While it is generally admitted that the coordination between cellular clocks is globally ensured by robust oscillations of internal synchronisers under the control of the central clock in a mono-directional way \cite{bass2010circadian}, recent studies have challenged this view and observed that circadian rhythm in liver cells (hepatocytes) can be maintained in a coherent and phase-organised manner in the absence of a functional SCN \cite{sinturel2021circadian}. 
	Therefore, the traditionally accepted one-directional hierarchy may require further refinement. 
	In particular, those observations suggest some degree of intercellular coupling between hepatocytes, independent of external cues. 
	Though the precise nature of this coupling remains unknown. 
	In a separate study  \cite{finger2021intercellular}, it has been proposed that intercellular coupling of molecular clocks may operate via a paracrine mechanism involving TGF-$\beta$ signalling. 
	TGF-$\beta$ is a secreted signalling molecule that seems to impact local circadian synchronisation, possibly regulating the expression of  the {\textit {Per2}} clock gene.
	Given that {\textit {Per2}} expression exhibits robust circadian oscillations \cite{takahashi2017transcriptional}, this observation suggests that intercellular coupling may be phase-dependent, being effective only during a specific stage of the circadian cycle.

	In this article, we propose to analyse local and phase-dependent intercellular coupling at the scale of a population of cells. 
	A {Kuramoto-like} model is used to represent circadian clocks in a cell population.
	In its discrete version, this model consists of a system of ordinary differential equations (ODEs) describing the dynamics of a population of coupled oscillators. 
	Each oscillator represents an oscillating cell, and is described by its phase indicating its progression through the cycle. 
	Hence, {the phase is interpreted as a continuous variable regulating} the production of some transcriptional regulators, and the intrinsic frequency represents the natural {speed of revolution} at which the oscillator would evolve in the absence of coupling.
	{In the seminal Kuramoto model \cite{kuramoto1984chemical},} oscillators interact with one another via a "global" coupling term (\textit{all-to-all} coupling where each oscillator is coupled to all the other oscillators at all time), which tends to reduce the phase differences between them. 
	This model is often used to analyse the synchronisation of an oscillators' population \cite{dorfler2011critical}. 
	For instance, in \cite{ermentrout1985synchronization}, it is shown that if the coupling strength passes a certain threshold, the population becomes phase-locked, meaning that all oscillators eventually evolve at the same frequency. 
	In \cite{ha2010complete}, authors show that when all oscillators share the same intrinsic frequency, their phases need only be sufficiently close at the start to achieve an asymptotically {synchronisation of the oscillators}.
	
	As a variant of the seminal Kuramoto model, we explore in the present {article} the possibility that intercellular communications may be phase-dependent. 
	This implies a looser coupling term, where the synchronising signal is only active during a specific interval of the cycle. 
	In a first step, we construct a Kuramoto-like model based on the biological cues cited above which formalises this phase-dependent coupling.
	We show that this relaxation of the Kuramoto model still generates asymptotic synchronisation (Theorem~\ref{thm:sync}).
	
	In a second step, a signal processing procedure is proposed to characterise quantitatively the oscillatory features observed in the experimental data obtained from {Per2}::luciferase reporter traces. 
	The code is available at \url{https://github.com/AnastasiaMARECHAL/InSyncPy}. 
	Then, the model is calibrated against experimental data.
	To this end, a reporter model is introduced, linking the theoretical model to experimental signals. 
A good agreement between the model and the data is achieved, and a range of values for the width of the coupling interval is identified, thereby supporting both the hypothesis of a coupling being active in a precise phase interval and the proposed model.
Finally, a comparative study of the calibration results between two experimental settings is performed leading to a better understanding of the dependencies between the coupling mechanisms and the reporter model as well as the phenotypic divergence and compensation mechanisms between the two settings.

\section{Modelling a network of interacting circadian clocks by a phase-local coupling in the Kuramoto model}

We consider a phase description of the dynamics of the circadian rhythm of a cell population.
The Kuramoto model was introduced in 1975 to model a network of coupled oscillators \cite{kuramoto1975international}.
The dynamics of the network is given by the following system on the oscillators' phases:
\begin{equation}
		\frac{d}{dt}\theta_i(t) = \Omega_i+\frac{K}{N}\sum\limits_{k = 1}^N\sin(\theta_k(t)-\theta_i(t)), \quad \text{for } t>0,\; i=1,\hdots, N,
	\label{eq:kuramoto}
\end{equation}
where $K>0$ is the coupling strength parameter and $N$ the total number of oscillators. 
The last term on the right-hand side is a global coupling term, each sinusoidal term in the sum drawing the phase $\theta_i$ closer to the other phases.
The coupling is thus all-to-all at any time $t$.
In some formulations \cite{dorfler2011critical}, variables $\theta_i$ are defined as angular variables on the torus $\mathds{R}_+/2\pi\mathds{Z}$. 
In the following, the system is formulated in $\left(\mathds{R}_+\right)^N$ to ease the presentation.
System \eqref{eq:kuramoto} leads to complete phase synchronisation under relatively large hypothesis, as described in \cite{ha2010complete}.
More precisely, it is sufficient that the phase distribution of the oscillators is not too widespread at initial time ($\max_{i,j} \vert\theta_i(0)-\theta_j(0)\vert < \pi$) in order to obtain an exponential speed of synchronisation of the oscillators for system \eqref{eq:kuramoto} in the case identical oscillators ($\Omega_i = \Omega$).

In the following, a population of $N$ oscillators is considered to model the time evolution of cellular circadian clocks in a population of hepatocytes.
The phase variable $\theta_i(t)$ represents the progression of a hepatocyte indexed by $i$ in the circadian cycle.
Furthermore, the population of cells is assumed identical (i.e. their intrinsic frequencies are equal), implying that, without intercellular coupling, each cell completes a full revolution of the circadian cycle in the same amount of time.
For the sake of clarity, we set the intrinsic frequency to $\Omega_i = 1$ for $i=1,\hdots,N$.

In the seminal Kuramoto model \eqref{eq:kuramoto}, the \textit{all-to-all} coupling is a strong hypothesis that does not seem to hold true in the case of interconnection of circadian clocks. 
Indeed, the cellular circadian clock is based on the succession of several qualitative stages, associated with specific transcriptional \cite{takahashi2017transcriptional} and metabolic \cite{bass2010circadian} mechanisms.
However, a consensus has been reached that it is more likely that cellular clocks are coupled during one or a fraction of the stages of the cycle, rather than for the entire duration of the cycle.
In particular, Finger et al. \cite{finger2021intercellular} proposed that paracrine signalling mechanisms could contribute to intercellular coordination. 
This signalling is known to influence the regulation of clock genes, including the expression of Per2. 
Per2 expression exhibiting circadian oscillations (see \textit{e.g.}, \cite{takahashi2017transcriptional}), it is reasonable to assume that cells adapt to incoming signals only during specific stages of the cycle.
Hence, we make the assumption that cells are susceptible to receive a synchronising message only during a specific phase interval.

By denoting $L>0$ the duration of a full revolution, the time evolution of the phases of a network of oscillators modelling the circadian rhythm of hepatocytes is given by:
\begin{equation}
	\left\{
	\begin{array}{ll}
		\dot\theta_i(t) &= 1 + \frac{K}{N} \varphi(\theta_i(t))\sum\limits_{k=1}^N \sin\left(\left(\theta_k(t)-\theta_i(t)\right)\frac{2\pi}{L}\right), \quad i=1,\hdots, N, \\
		\theta_i(0) &= \theta_i^0\in (0,L), \quad i=1,\hdots, N,
	\end{array}
	\right.
	\label{eq:adaptive_kura}
\end{equation}
where $\theta_i^0$ is the initial phase of the oscillator $i$ and $\varphi$ designates a $L$-periodic, {compactly supported on the restriction $[0, L]$,} coupling function.
Here, $\varphi$ modulates the coupling strength of the oscillator depending on whether or not its phase is within the correct interval.
Indeed, we assume that $\varphi$ has compact support in the phase domain and
we denote by $S\subsetneq[0,L]$ its support
such that  
\begin{equation}
	S\supset I_\theta \defeq 
	[\theta_{\min};\theta_{\max}], 
	\quad \text{and} \quad \Delta_\theta \defeq \theta_{\max}-\theta_{\min}.
	\label{eq:delta_theta}
\end{equation}
where $0<\theta_{\min}<\theta_{\max}<L$ are two given values (in CT), $I_\theta$ is the phase interval with non-negligible coupling strength and $\Delta_\theta$ is the length of this interval.
For instance, if we consider the signal to be mediated through the direct modulation of Per2 expression, one may choose $\theta_{\min}=13$ or $14$ (in CT) and $\theta_{\max}=20$ or $21$ (in CT, see \cite{takahashi2017transcriptional}).
Moreover, we make several assumptions on the coupling function $\varphi$ in order to ensure the following:
\begin{itemize}
	\item the coupling operates during the stages associated to the window $I_\theta$,
	\item the coupling function is bounded, i.e. the intercellular coupling is seen as a perturbation on the speed at which the cell completes one revolution of its circadian cycle. 
\end{itemize}
The following hypothesis formalises all these assumptions.
\begin{Hyp}[Phase-local coupling] Let $\varphi\in\mathcal{C}^1(\mathds{R}_+)$ be a $L$--periodic function $(\varphi(x+L)=\varphi(x))$ and $\text{supp}(\varphi)= S/ (L\mathds{N}) \subsetneq \mathds{R}^+$.
Moreover, let
\begin{itemize}
	\item  $\varphi$ be $L_\varphi$--Lipschitz continuous, $L_\varphi >0,$ $|\varphi '(x)|\leq L_\varphi $  and $0\leq \varphi(x)\leq 1, \quad \forall x\in \mathds{R}_+,$
	\item $\exists C_\varphi>0, \quad 0<C_\varphi\leq \varphi(x)\leq1$ for $x\in\mathring{I}_\theta/(L\mathds{N})$ (cf. \eqref{eq:delta_theta}).
\end{itemize}
\label{hyp:phi}
\end{Hyp}
\noindent As an illustration, the following figure gives an admissible shape for the coupling function  $\varphi$.
\noindent
\begin{minipage}{0.495\textwidth}

	\begin{tikzpicture}
		\begin{axis}[
			width=\textwidth,
			height=0.7\textwidth,
			xlabel={$\theta$},
			domain=0:4.1,
			samples=300,
			axis lines=left,
			ymin=0, ymax=1.1,
			xmin=0, xmax=4.1,
			clip=false,
    		xtick={0,1.23,2.77,4},
    		xticklabels={$0$,$\theta_{\min}$,$\theta_{\max}$,$L$},
    		ytick={0,0.2,1},
    		yticklabels={$0$,$C_\varphi$,$1$}
			]
			
			
			
			\addplot[thick, blue]
			{abs(x-2)<1 ? exp(1)*exp(-1/(1-(x-2)^2)) : 0};
			
			\node at (axis cs:3.5,0.1) {$\varphi(\theta)$};
			
			\draw[<->, thick, red]
			(axis cs:1.23,0) --
			(axis cs:2.77,0)
			node[midway, below] {$\Delta\theta$};
			
			\draw[dashed]
    		(axis cs:0,0.2) --
    		(axis cs:2.77,0.2);
    		
 			\draw[dashed]
    		(axis cs:2.77,0) --
    		(axis cs:2.77,0.2);   
    		
    		\draw[dashed]
    		(axis cs:1.23,0) --
    		(axis cs:1.23,0.2);
    			
		\end{axis}
	\end{tikzpicture}
\end{minipage}
\begin{minipage}{0.495\textwidth}
The function $\varphi$ is bounded by below and above on $I_\theta$.
On $[0,L]\setminus S$, $\varphi = 0$ which implies no coupling between the oscillators.
In these regions, the oscillator $\theta_i$ evolves freely following its intrinsic frequency.
On the interval $S\setminus I_\theta $, there is a small coupling strength since $C_\varphi>\varphi\geq 0$ and we consider it negligible.
\end{minipage}

System (\ref{eq:adaptive_kura}) can be seen as a specific \emph{adaptive Kuramoto model}, where parameter $K$ is generalised as a function of the phase $K(\theta)$. 
The generalisation through the adaptative Kuramoto model was initially introduced to formalise processes where interaction strengths are not fixed parameters but evolve in time due to learning, plasticity, or feedback from the oscillators' states (cf. \cite{ha2016synchronization, hoppensteadt2000pattern, seliger2002plasticity}). 
The {adaptive Kuramoto model} is particularly well suited for studying a network of coupled circadian clocks, since intercellular coupling can depend on physiological state (the stages of the circadian clock), signalling efficacy, and regulatory feedback \cite{finger2021intercellular}. 
Hence, for circadian cell networks, adaptive coupling offers a more realistic representation of how the tissue-level rhythm emerges from interacting cellular oscillators, at the expense of introducing additional non-linearities in the phase dynamics and increasing the analysis and computational costs.

To our knowledge, one of the original motivation behind the introduction of adaptive Kuramoto model was to obtain a coupling strength that depends of the location of the phases and more specifically their distance between each other (cf. \cite{ha2016synchronization, seliger2002plasticity}).
This coupling term increases more when the two oscillators are near to be phase synchronised.
In the present study, this is not the case. 
The phase dependent coupling term is kept at 0 outside a specific interval, allowing the coupling to be active only momentarily.
Here, a structure for the function $K(\theta)$ is imposed, through the function $\varphi$, directly inspired from biological considerations.
We make an additional hypothesis to ensure monotonicity for the progression in the circadian cycle. 
\begin{Hyp}[Coupling strength] Let the parameter $K$ associated to the coupling strength between oscillators in the network such that
	$K \in (0,1). $
	\label{hyp:K}
\end{Hyp}
\noindent Under this hypothesis, and given that $\varphi\leq 1$, the following inequality holds $$\frac{K}{N}\varphi(\theta_i)\sum_{k=1}^N\sin(\theta_k-\theta_i)> -1, \quad i \in\{1, ..., N\},$$
implying
$$\frac{d}{dt} \theta_i(t) > 0, \quad\forall \forall i = 1, \hdots, N,\, t\in\mathds{R}_+. $$ 
\noindent Hence, Hypotheses \ref{hyp:phi} and \ref{hyp:K} ensure that the phases are always increasing.\\

\section{Asymptotic phase synchronisation of the model}
For the sake of clarity and without loss of generality, we fix the value of the period $L = 2\pi$ in this section.
Hence, we consider the following system :

\begin{equation}
	\left\{
	\begin{array}{rl}
		\frac{d}{dt}\theta_i(t) &= 1 + \frac{K}{N}\varphi(\theta_i(t)) \sum_{j=1}^N \sin\left(\theta_j(t)-\theta_i(t)\right),\quad i=1,\hdots,N,\\
		\theta_i(0) &= \theta_i^0\in (0,2\pi),\quad i=1,\hdots,N.
	\end{array}
	\right.
	\label{eq:math_model}
\end{equation}

As a first step, preliminary properties of the solution, including well-posedness, are established. 
An order parameter (phase diameter) is then introduced and shown to be non-increasing in time. 
Finally, we prove that system \eqref{eq:math_model} achieves complete phase synchronisation asymptotically. 
Compared to the seminal Kuramoto model, the coupling term is weaker but the speed of convergence toward the synchronised state is also slower.

\subsection{Properties of the oscillators' phases}
In the following, we prove in Proposition \ref{prop:no_crossing} that the solution of system \eqref{eq:math_model} is well-posed, and that the phase trajectories do not intersect with one another in finite time. 
Then, we establish a monotony property (Lemma \ref{le:div_phase}) stating that under the hypotheses \ref{hyp:phi} and \ref{hyp:K}, the phase of each oscillator is strictly increasing.

\begin{prpstn}[Well-posedness] Suppose Hypotheses \ref{hyp:phi} and \ref{hyp:K} hold.
Consider $\Theta = \left(\theta_i\right)_{i=1,\hdots,N}$ and $\Theta^0 = \left(\theta_i^0\right)_{i=1,\hdots,N}$ such that 
$$\theta_i^0 \neq \theta_j^0, \quad \text{for } i \neq j. $$
Then there exists a unique global solution for system \eqref{eq:math_model} on $\mathds{R}_+$, with $\theta_i \in \mathcal{C}^1(\mathds{R}_+)$ for all $i=1,\dots,N$.
Moreover the following property holds:
$$ \theta_i(t) \neq \theta_j(t), \quad \text{for } i \neq j,\,i,j =1,\hdots,N  ,\, \forall t > 0. $$
\label{prop:no_crossing}
\end{prpstn}
\begin{proof}{\textit{(Proposition \ref{prop:no_crossing})}}
The well-posedness of system \eqref{eq:math_model} is a direct consequence of the Cauchy–Lipschitz theorem, since the vector field of the system is globally Lipschitz continuous under Hypothesis \ref{hyp:phi}.
Moreover, let $\Theta$ be a solution of system \eqref{eq:math_model}, with initial conditions satisfying $\theta_i^0\in (0,2\pi)$ for all $i=1,\dots,N$ and 
$$\theta_i^0 \neq \theta_j^0, \quad \text{for } i \neq j.$$
For any pair $(i,j) \in \{1,\dots,N\}^2$ with $i \neq j$, we denote the phase difference as
$$\phi_{i,j}(t) \defeq \theta_i(t) - \theta_j(t).$$
Hence, by assumption, $|\phi_{i,j}(0)| > 0$ for all $i\neq j$.
Then, for all $t>0$,
\begin{align*}
\frac{d}{dt}|\phi_{i,j}| =& sign(\phi_{i,j})\frac{d}{dt}\phi_{i,j},\\
=& sign(\phi_{i,j})\frac{K}{N}\sum_{k = 1}^N\varphi(\theta_i)\sin(\theta_k-\theta_i)
- \varphi(\theta_j)\sin(\theta_k - \theta_j),\\
=& sign(\phi_{i,j})\frac{K}{N}\sum_{k=1}^N(\varphi(\theta_i) - \varphi(\theta_j))\sin(\theta_k - \theta_i)
+ \varphi(\theta_j)(\sin(\theta_k - \theta_i) - \sin(\theta_k - \theta_j)).\\
\end{align*}
Thanks to Hypothesis \ref{hyp:phi}:
$$ - L_\varphi|\phi_{i,j}|\leq \varphi(\theta_i) - \varphi(\theta_j)\leq L_\varphi|\phi_{i,j}|,$$
and since the sine function is $1$-Lispchitzian:
$$-|\phi_{i,j}|\leq \sin(\theta_k - \theta_i) - \sin(\theta_k-\theta_j)\leq |\phi_{i,j}|.$$
Thus, 
$$ \frac{d}{dt}|\phi_{i,j}|\geq  - K(L_\varphi + 1)|\phi_{i,j}|.$$
Using the Grönwall's inequality, we obtain for all $t>0$
$$|\phi_{i,j}(t)|\geq |\phi_{i,j}(0)|e^{- K(L_\varphi + 1)t}.$$
Then $|\phi_{i,j}(t)|>0$ for all $t>0$.
\end{proof}

Proposition \ref{prop:no_crossing} states that the trajectories never intersect. 
This suggests that complete synchronisation can only be achieved asymptotically under a widespread phase's initial distribution.
Another useful characteristic of the system in order to prove the asymptotic phase synchronisation is the monotony of the oscillators' phases.

\begin{lmm}
	Let Hypotheses \ref{hyp:phi} and \ref{hyp:K} hold and let $\{\theta_i\}_{i  =1}^N$ be the global solution of \eqref{eq:math_model}. Then $\theta_i$ is strictly increasing and diverge to $+\infty$ as $t\to\infty$ for all $i\in\{1, ..., N\}$.
	\label{le:div_phase}
\end{lmm}
\begin{proof}{\textit{(Lemma \ref{le:div_phase})}}
	Let $t>0$, 
	$$\frac{d}{dt}\theta_i (t) = 1 + \varphi(\theta_i(t))\frac{K}{N}\sum_{k = 1}^N\sin(\theta_k(t) - \theta_i(t))\geq 1 - K > 0$$
	because $K<1$ (c.f. Hyp \ref{hyp:K}) and $\varphi(x)\geq 0$ (c.f. Hyp \ref{hyp:phi}). 
	Thus, each $\theta_i$ is strictly increasing.
	Also, for $i\in\{1, ..., N\}$ 
	$$\theta_i(t) - \theta_i^0 \geq (1 - K)t.$$
	Since $\theta_i$ is continuous on $\mathds{R}^+$, and $1 - K>0$, by comparison, $\lim_{t\to\infty}\theta_i(t)=+\infty$ for all $i\in\{1, ..., N\}$.
\end{proof}	

\subsection{The phase diameter as a synchronisation's order parameter}
Since the trajectories never intersect (cf. Prop. \ref{prop:no_crossing}), extremal phases remain the same over time, allowing the definition of the phase diameter as a difference between extremal phases (see Eq. \eqref{eq:phase_diameter} below).
This specific diameter was previously introduced in \cite{ha2010complete} to obtain phase and frequency synchronisation estimates in the seminal Kuramoto model. 
We use this definition to extend the study of synchronisation to the modified Kuramoto model with the "phase-local" coupling, which is better aligned with the biological context being modeled.\\
Let $m\in\{1, \hdots, N\}$ be the index such that 
$$\theta_m^0 \defeq \min_{1 \leq i \leq N} \theta_i^0$$
 and $M\in\{1, \hdots, N\}$ be the index such that $$\theta_M^0 \defeq \max_{1 \leq i \leq N} \theta_i^0.$$
Under the assumption of Proposition \ref{prop:no_crossing}
\begin{equation}
\theta_m(t) = \min_{1 \leq i \leq N} \theta_i(t)\qquad \mathrm{and} \qquad\theta_M(t) = \max_{1 \leq i \leq N} \theta_i(t), \text{ for all } t>0.
\label{eq:thetamM}
\end{equation}
To quantify the distribution of the phases of the oscillators, we introduce the phase diameter as follows.
\begin{dfntn}[Phase diameter]
We denote the phase diameter as
\begin{equation}
D_\theta(t) \defeq \theta_M(t) - \theta_m(t), \quad \forall t>0,
\label{eq:phase_diameter}
\end{equation}
and we denote by $D^0$ the initial phase diameter 
$$D^0 \defeq \theta_M^0 - \theta_m^0. $$ 
\end{dfntn}
\noindent Now, the synchronisation of the system \eqref{eq:math_model} can be characterised as follows.
\begin{dfntn}[Phase synchronisation]
\label{def:sync}
We consider that system \eqref{eq:math_model} tends toward phase synchronisation if
$$\lim_{t\to\infty}|\theta_i(t) - \theta_k(t)| = 0, \quad \forall i, k\in \{1, \dots, N\}.$$
Or, equivalently, system \eqref{eq:math_model} is synchronising over time if the phase diameter tends to zero, i.e.,
$$\lim_{t\to\infty}D_\theta(t) = 0.$$
\end{dfntn}
\noindent Definition \ref{def:sync} states that the system is synchronising if the phases' distribution concentrates towards a Dirac mass asymptotically.
A monotony property of the phase diameter \eqref{eq:phase_diameter} is obtained subsequently.
It allows to obtain uniform bounds of the phase diameter.

\begin{prpstn}[Non-increasing phase diameter]
Suppose Hypotheses \ref{hyp:phi} and \ref{hyp:K} hold and let $\left(\theta_i\right)_{i  =1}^N$ be the global solution of \eqref{eq:math_model} with initial size distribution satisfying 
$D^0<\pi.$
Then $D_\theta$ is nonincreasing and subsequently
$$0\leq D_\theta(t)\leq D^0, \quad  \forall t>0.$$ 
\label{prop:phase_decay}
\end{prpstn}
\begin{proof}{\textit{(Proposition \ref{prop:phase_decay})}}
Let $t^*>0$ such that 
$$t^* = \inf\{t\geq 0, D_\theta(t) = \pi\}.$$ 
Since $D^0<\pi$ and $D_\theta$ continuous, this implies $D_\theta(t)<\pi$ for all $t\in[0, t^*)$.
Let $t\in[0, t^*)$, 
\begin{align*}
	\dt D_\theta(t) &= \dt \theta_M(t) - \dt \theta_m(t),\\
	& = \frac{K}{N}\sum_{k=1}^N\big[\varphi(\theta_M(t))\sin(\theta_k(t)-\theta_M(t))-\varphi(\theta_m(t))\sin(\theta_k(t) - \theta_m(t))\big].
\end{align*}
We have $$\sin(\theta_k(t)-\theta_M(t))\leq 0$$
because $-\pi< -D_\theta(t)\leq \theta_k(t) - \theta_M(t)\leq 0$. 
Similarly: $$\sin(\theta_k(t)-\theta_m(t))\geq 0.$$ 
Since $\varphi\geq 0$ (cf. Hyp. \ref{hyp:phi}), we obtain
$$\dt D_\theta(t)\leq 0,\quad  \forall t\in[0,t^*)$$
and then $$D_\theta(t)\leq D^0<\pi, \quad  \forall t\in[0,t^*).$$
Now, suppose $0<t^*<\infty$, then
$$\pi = D_\theta(t^*) = \lim_{t \to t^* \atop t < t^*}D_\theta(t).$$
This leads to the following contradiction:
$$\pi = D_\theta(t^*) = \lim_{t \to t^* \atop t < t^*}D_\theta(t)\leq D^0<\pi.$$
Hence, 
it implies that $D_\theta$ is nonincreasing and bounded by  $$0\leq D_\theta(t)\leq D^0<\pi,\quad \forall t\geq 0.$$
\end{proof}

\subsection{Asymptotic phase synchronisation}

Now, we look at the evolution of the phase diameter on the initial cycle (i.e. the first revolution of the oscillators) and iterate its study across cycles.
We show that, under the assumption $D^0<\Delta_\theta$, there is a non-empty time interval during which all oscillators are within the coupling interval, and using arguments similar to those introduced in \cite{ha2010complete}, an exponential decay of the diameter on this interval is obtained.
Then, we show that the size of this time interval is uniformly lower-bounded across cycles, which allows us to obtain a geometric decay of the diameter across cycles and thus to conclude that system \eqref{eq:math_model} tends to a phase synchronised state.

\subsubsection{Contraction of the phase diameter on the first cycle}

The first step is to quantify the decay of the phase diameter $D_\theta$ during the initial cycle, as well as the time window during which all oscillators are in the coupling interval {$I_\theta$} (see \eqref{eq:delta_theta}).
We define the two following characteristic features: 
\begin{itemize}
    \item the first time when all oscillators are coupled during the initial cycle denoted by $T_0$ with $\theta_m(T_0)= \theta_{min},$
    \item the time span during the initial cycle when all oscillators are coupled denoted by $\delta_0$ with $\theta_M(T_0 + \delta_0)= \theta_{max}$.
\end{itemize}
Using these two characteristic features, we obtain the following decay property on the phase diameter after the first revolution of the cycle.

\begin{prpstn}
	Suppose Hypotheses \ref{hyp:phi} and \ref{hyp:K} hold and let $\{\theta_i\}_{i  =1}^N$ be the global solution of \eqref{eq:math_model} with initial phase distribution satisfying $D^0<\min(\pi, \Delta_\theta)$.
	Suppose that $\theta_m^0\leq \theta_{\min}$.
	Using notation introduced in \eqref{eq:thetamM}, let $T_1$ and $\delta_1$ designate respectively the time and the time span such that $$\theta_m(T_1) = \theta_m(T_0) + 2\pi \quad \mathrm{and}\quad \theta_M(T_1+\delta_1) = \theta_M(T_0 + \delta_0)+ 2\pi.$$
	Then, we have for $t\in [T_0, T_0 + \delta_0]$
	$$D_\theta(T_1)\leq D_\theta(T_0 +\delta_0)\leq D_\theta(T_0)e^{-\alpha K{C_\varphi}(t - T_0)}, \quad\mathrm{with}\quad \alpha \defeq \frac{\sin(D^0)}{D^0}>0.$$
	\label{prop:cycle_0}
\end{prpstn}

\begin{rmrk}
	In Proposition \ref{prop:cycle_0}, we suppose that $\theta_m^0 < \theta_{min} $, i.e. the initial phase distribution support is not included in the coupling phase interval. 
Otherwise, the same result can be obtained by shifting to the second cycle and redefining $T_i, \; \delta_i $, $i=0,\; 1$ accordingly.
This assumption does not affect the decay estimate on the first cycle. 
	Moreover, one can note that on the interval $[T_0, T_0 +\delta_0]$, all phases are between $\theta_{\min}$ and $\theta_{\max}$. Hence, similar arguments as in \cite{ha2010complete} can be used to obtain the exponential decay of the phase diameter $D_\theta$.
\end{rmrk}

\begin{proof}{\textit{(Proposition \ref{prop:cycle_0})}}
The proof consists of two steps. First, we verify that all oscillators remain in the coupling interval during a positive time interval. 
Then, an exponential decay estimate for the diameter is derived.\\
\noindent\textit{Step 1: Non trivial coupling time span.}\\
\noindent Since $\theta_m(t)$ is strictly increasing with respect to time and tends to infinity (c.f. Lemma \ref{le:div_phase}), $T_0$ and $T_1$ are uniquely defined and $T_0<T_1$.
Since $D^0<\Delta_\theta$, $\theta_m(T_0) = \theta_{\min}$ and $\sin(\theta_k(t) - \theta_M(t))\leq 0$ for $t\geq 0$, we have $\theta_M(T_0)< \theta_{\max}.$\\
Hence, the time span $\delta_0$ such that $\theta_M(T_0 + \delta_0)= \theta_{\max}$ is such that $\delta_0>0$.
	Also, $\theta_M$ is increasing and $\lim_{t\to\infty}\theta_M = +\infty$ (c.f. Lemma \ref{le:div_phase}), so $\delta_0$ is unique.
	Finally, since $\theta_m(t)<\theta_M(t)$ for all $t>0$ and $$\theta_M(T_0+\delta_0) = \theta_{\max}<\theta_{\min}+2\pi = \theta_m(T_1)$$ then, by continuity and strict monotony, $T_1 > T_0+\delta_0$.\\
\noindent\textit{Step 2: Exponential decay of $D_\theta$ on $[T_0, T_0 +\delta_0]$.}\\
	For $t<T_0$, the function $\varphi$ is null or arbitrary small ($\varphi(x)<C_\varphi \ll 1$ for $x\in [0;2\pi]\setminus I_\theta$) hence the phase diameter stays almost constant. 
	Let $t\in[T_0,\ T_0+\delta_0]$,  since $\theta_M(t), \; \theta_m(t) \in [\theta_{\min}, \theta_{\max}]$  and {$\varphi(\theta_m(t))\geq C_\varphi,$ $\varphi(\theta_M(t)) \geq C_\varphi$ (cf. Hyp. \ref{hyp:phi}), }
	we get that
	\begin{align*}
		\frac{d}{dt}D_\theta(t) &= \frac{K}{N}\varphi(\theta_M(t))\sum_{k=1}^N\sin(\theta_k(t)-\theta_M(t)) - \frac{K}{N}\varphi(\theta_m(t))\sum_{k=1}^N\sin(\theta_k(t)-\theta_m(t)), \\
		&{\leq \frac{KC_\varphi}{N}}\sum_{k=1}^N\sin(\theta_k(t)-\theta_M(t)) - \sin(\theta_k(t)-\theta_m(t)).
	\end{align*}
	Now, $\theta_k(t) - \theta_M(t)\in[-\pi, 0]$ for all $t\geq 0$ and $k\in\{0, \hdots, N\}$ since $\theta_M(t)\geq \theta_k(t)$, thanks to Proposition \ref{prop:phase_decay}, we get 
	$$|\theta_k(t)-\theta_M(t)|\leq D_\theta(t)<D_0.$$  
Since the sine function is concave on $[-\pi, 0]$, we obtain for all $k\in\{1, ..., N\}$,
$$\sin(\theta_k(t) - \theta_M(t))\leq \frac{\sin(D^0)}{D^0}(\theta_k(t) - \theta_M(t)).$$
Similarly, $\theta_k(t) - \theta_m(t)\in[0, \pi]$ for all $t\geq 0$ and the sine function is convex on $[0, \pi]$, then, for all $k\in\{1, ..., N\}$,
	 $$\sin(\theta_k(t) - \theta_m(t))\geq \frac{\sin(D^0)}{D^0}(\theta_k(t) - \theta_m(t)).$$
Denoting $\alpha = \frac{\sin(D_0)}{D_0}$, we have
	\begin{align*}
		\frac{d}{dt}D_\theta(t) &\leq \frac{KC_\varphi}{N}\sum_{k=1}^N\sin(\theta_k(t)-\theta_M(t)) - \sin(\theta_k(t)-\theta_m(t)),\\
		&\leq \alpha\frac{KC_\varphi}{N}\sum_{k=1}^N(\theta_k(t)-\theta_M(t)-\theta_k(t)+\theta_m(t)),\\
		&\leq -\alpha KC_\varphi(\theta_M(t)-\theta_m(t)) = -\alpha KC_\varphi D_\theta(t).
	\end{align*}
	Using Grönwall’s lemma, we obtain for $t\in [T_0, T_0 + \delta_0]$ that 
	$$D_\theta(t)\leq D_\theta(T_0)e^{-\alpha K{C_\varphi} (t-T_0)}.$$
	Since $D_\theta$ is nonincreasing (c.f. Lemma \ref{prop:phase_decay}) and $T_0<T_1$, we have:
	$$D_\theta(T_1)\leq D_\theta(T_0 + \delta_0)\leq D_\theta(T_0)e^{-\alpha K{C_\varphi}(t-T_0)}, \quad \text{for } T_0\leq t\leq T_0 + \delta_0.$$
\end{proof}

\subsubsection{Phase diameter decay and iterative contractions across cycles}

Now, we show that $D_\theta$ converges to 0, which leads the system to asymptotic phase synchronisation.
First, we define iteratively $T_n$ and $\delta_n$, the time when all oscillators are in the coupling interval during the $n+1$ cycle and the time span during which all oscillators are coupled respectively, 
\begin{equation}
\begin{array}{l}
	\theta_m(T_n)\defeq \theta_m(T_{n-1})+ 2\pi,\quad n\in \mathds{N}^*,\\ 
	\theta_M(T_n + \delta_n)  \defeq \theta_M(T_{n-1}+ \delta_{n-1}) + 2\pi,\quad n\in \mathds{N}^*.
\end{array}
	\label{eq:tehtamM_n}
\end{equation}

\begin{prpstn}
	Suppose Hypotheses \ref{hyp:phi} and \ref{hyp:K} hold and let $\{\theta_i\}_{i  =1}^N$ be the global solution of \eqref{eq:math_model} with initial phase distribution satisfying $D^0<\min(\pi, \Delta_\theta)$.
	Suppose that $\theta_m^0\leq \theta_{\min}$.
	Then, 
	using the notation introduced in \eqref{eq:tehtamM_n},  we have
	$$D_\theta(T_{n+1})\leq D_\theta(T_n)e^{-\alpha K{C_\varphi}(t - T_n)},\quad \forall  n\in\mathds{N}^*,\quad \forall t\in[T_n, T_n+\delta_n].$$
	\label{prop:cycle_n}
\end{prpstn}

\noindent The proof of the contraction of the phase diameter stated in Proposition \ref{prop:cycle_n} uses the same arguments as in the proof of Proposition \ref{prop:cycle_0}.
We propose to give a sketch of the proof. First, we use the monotony of each $\{\theta_i\}_{i =1}^N$ to prove that the sequences $(T_n)_{n\in\mathds{N}^*}$ and $(\delta_n)_{n\in\mathds{N}^*}$ are well defined and that $T_{n+1}>T_n+\delta_n\geq T_n$. 
Moreover, since the initial phase diameter verifies $D^0<\Delta_\theta$ and the phase diameter $D_\theta$ is non-increasing, then $\delta_n>0$ for all $n\in\mathds{N}^*$. 
Hence, the intervals $[T_n, T_n + \delta_n]$ are non-empty for all $n\in\mathds{N}^*$. 
The exponential decay of the phase diameter is finally obtained by applying the same reasoning as in the proof of the Proposition \ref{prop:cycle_0}, switching $T_0$ by $T_n$, $T_1$ by $T_{n+1}$ and $\delta_0$ by $\delta_n$. 
The result holds for all $n\in\mathds{N}^*$.\\
Proposition \ref{prop:cycle_0} gives an exponential decay during a specific time interval of length $\delta_0$ and Proposition \ref{prop:cycle_n} extends this result to intervals of length $\delta_n$ for every $n\in\mathds{N}^*$.
The following lemma gives a uniform lower bound of this time span $\delta_n$.

\begin{lmm}
	\label{lemma:borne_delta}
	Let the same assumptions as in Proposition \ref{prop:cycle_0} hold, then 
	$$\delta_n\geq \delta>0, \quad \forall n\in\mathds{N}, $$ 
	where $\delta = \frac{\Delta_\theta - D^0}{1+K}.$
\end{lmm}

\begin{proof}{\textit{(Lemma \ref{lemma:borne_delta}})}
	Let $n\in\mathds N$.
	For $t\geq 0$, using Hypothesis \ref{hyp:phi} and the uniform upper bound on the sine function, we have that 
	$$\frac{d}{dt}\theta_M(t)  =  1 + \varphi(\theta_M(t))\frac{K}{N}\sum_{k=1}^N\sin\left(\theta_k(t)-\theta_M(t)\right)\leq 1+K.$$
	Then, for $t\in [T_n, T_n+\delta_n]$, we introduce the comparison function 
	\begin{equation}
		g(t)=\theta_M(T_n)+(1+K)(t - T_n),
		\label{eq:comparison_fct}
	\end{equation} which satisfies $g(T_n) = \theta_M(T_n)$ and $\theta_M(t)\leq g(t)$.
	
	Now, we look at the time span $t_n$ such that $g(T_n+t_n)=\theta_{\max}+2\pi n$. 
	Since $\dt \theta_M(t) \leq \dt g(t)$, we have  
	$\delta_n\geq t_n.$
	With Eq. \eqref{eq:comparison_fct}
	$$t_n = \frac{\theta_{\max}+2\pi n-\theta_M(T_n)}{1+K}.$$
	Since $\theta_m(T_n) = \theta_{\min}+2\pi n$ and $D_\theta(T_n)\leq D^0$  (c.f. Lemma \ref{prop:phase_decay}), we obtain that
	\begin{align*}
		\theta_M(T_n) &= \theta_m(T_n)+D_\theta(T_n)\leq \theta_{\min} + 2\pi n + D^0.\\
	\end{align*}
	Therefore
	$$\delta_n\geq t_n \geq \delta \defeq \frac{\Delta_\theta - D^0}{1+K}>0.$$
\end{proof}

Finally, the following Theorem establishes the result on the convergence toward a synchronised state for the modified Kuramoto model.
\begin{thrm}
	Suppose Hypotheses \ref{hyp:phi} and \ref{hyp:K} hold and let $\{\theta_i\}_{i  =1}^N$ be the global solution of \eqref{eq:math_model} with initial phase distribution satisfying $D^0<\min(\pi, \Delta_\theta)$.
	Suppose that $\theta_m^0\leq \theta_{\min}$.
	Then, we obtain
	$$D_\theta(T_{n+1}) \leq D_\theta(T_n) e^{-\alpha K {C_\varphi}\delta}, \qquad \forall n\in\mathds{N}.$$
	In particular, the sequence $(D_\theta(T_n))_{n\in\mathds{N}}$ decays to 0 faster than a geometric sequence with ratio $ e^{-\alpha K {C_\varphi}\delta}$.
	\label{thm:sync}
\end{thrm}

\begin{proof}{\textit{(Theorem \ref{thm:sync})}}
	Let $n\in\mathds N$.
	Recall the result in Proposition \ref{prop:cycle_n}, for all $t\in[T_n, T_n+\delta_n]$
	$$D_\theta(T_{n+1})\leq D_\theta(T_n)e^{-\alpha KC_\varphi(t - T_n)}.$$
	Especially, for $t = T_n + \delta_n$, we have 
	$$D_\theta(T_{n+1})\leq D_\theta(T_n)e^{-\alpha KC_\varphi\delta_n}.$$
	Also, since $\delta_n\geq \delta>0$ (see Lemma \ref{lemma:borne_delta}), then,
	$$D_\theta(T_{n+1})\leq D_\theta(T_n)e^{-\alpha KC_\varphi\delta}.$$
	Thus, $(D_\theta(T_n))_{n\in\mathds{N}}$ is a geometrically decreasing sequence with common ratio
	$$q := e^{-\alpha KC_\varphi\delta}\in]0,1[$$
	since $\alpha K C_\varphi \delta > 0$.
	Therefore, $$D_\theta(T_n)\to 0\text{ as }n\to\infty.$$
	Also, since $T_{n+1}>T_n+\delta_n$, an induction argument yields $$T_n>T_0+n\delta,$$
	and then $T_n\to+\infty$ as $n\to+\infty$.
	Finally, let $t>T_0$ and choose $n\in\mathds N$ such that $t\in[T_n, T_{n+1}]$, thus $$0\leq D_\theta(t)\leq D_\theta(T_n).$$
	As $t\to+\infty$, we have necessarily $n\to+\infty$ and since $D_\theta(T_n)\to 0$ as $n\to+\infty$, $D_\theta(t)\to 0$ as $t\to+\infty$.
\end{proof}

\section{Processing and analysis of experimental data}
Having established the main synchronisation result in the previous section, we now evaluate the relevance of the model by comparing it with experimental data. 
We begin by presenting the {Per2}::luciferase reporter dataset before introducing the signal processing approach used to extract quantitative features that will help confronting mathematical model and observations in the next section.\\
Quantitative characterisation of oscillatory features in experimental signals of circadian rhythms is crucial for uncovering underlying regulatory mechanisms and enabling model-driven hypothesis testing. 
Accurately quantifying key oscillatory parameters, such as the temporal evolution of the period and amplitude, sheds light on the entrainment dynamics and intercellular coupling governing circadian clock synchronisation or desynchronisation in ex vivo experiments. 
More generally, a rigorous quantification of oscillatory signals improves the preciseness of the interpretations, and thus enhances both mechanistic insight and translational application.\\
Specifically, this section details the data processing pipeline used to extract the instantaneous period and amplitude from the reporter recordings.  
Finally, a preliminary analysis of these quantities is provided, highlighting the main dynamical features of experimental data prior to model calibration.

\subsection{Experimental time series}\label{sec:bio_data}	
In this study, we use spheroids composed of MMHD3 hepatocytes as an experimental model system to monitor phase localisation within the circadian cycle. 
We utilised cells with two different periods: control cells, denoted as wild-type (WT), which exhibit a period close to 24 h, and Cry2KO cells, in which the {Cry2} clock gene has been knocked out, exhibiting a longer period.
Experiments were carried out on populations of $n=61$ spheroids containing approximatively  2000 cells each.	
In total, there are 20 datasets for WT and 15 for Cry2KO spheroids.
	At the beginning of the experiments, a 30 min pulse of dexamethasone is applied to chemically realign the circadian clocks of the cells within the spheroids.
	The MMHD3 cells display a strong contact inhibition when assembled in spheroids and therefore do not divide. 
Furthermore, the medium for each experiment is sufficiently rich to sustain cell activity throughout its duration, meaning that cell death can be assumed to be negligible. 
To precisely and quantitatively retrieve information regarding the circadian clock, the reporter gene {Per2}:luciferase, stably integrated into both WT and Cry2KO cells, is used as a proxy to record real-time photon emission by luciferase upon oxydation of the luciferin substrate, when the Per2 gene is expressed.\\
\begin{figure}[h!]
			\begin{center}
			\includegraphics[width=0.8\textwidth]{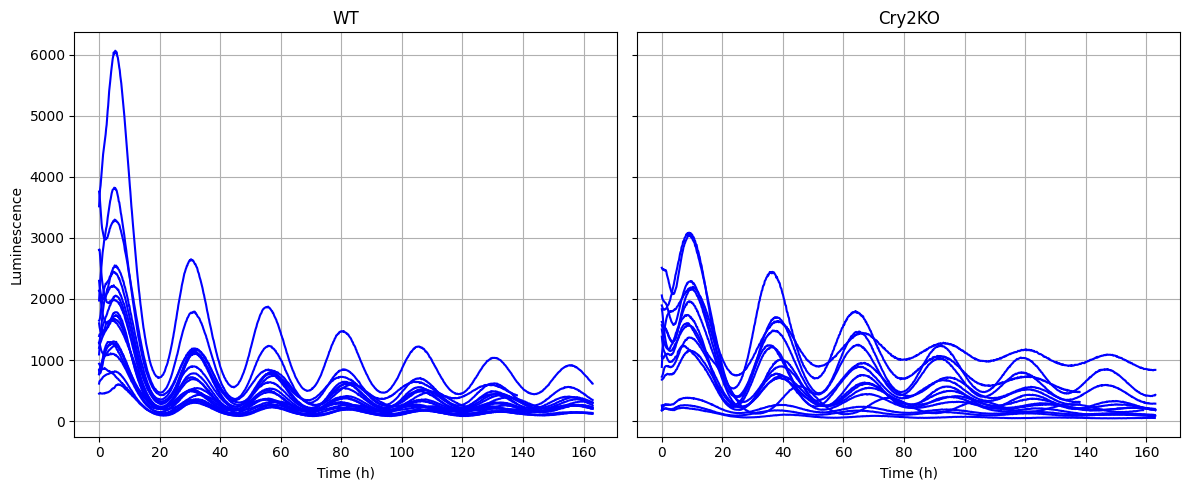}
			\caption{ Raw luminescence recordings for the WT (left) and Cry2KO (right) datasets. Each trace corresponds to one dataset of 61 spheroids of approximately 2000 cells.} 
			\label{fig:data_brut}
		\end{center}
\end{figure}
Figure \ref{fig:data_brut} shows the raw luminescence recordings for the WT (left) and Cry2KO (right) datasets, where each trace corresponds to one dataset. 
Oscillations of significant amplitude indicate that the {Per2} genes of all cells are transcribed synchronously, demonstrating coordination within the cell population.
Conversely, a flat trace indicates a desynchronised population.
Although more datasets are available for the WT populations, the Cry2KO populations appear to exhibit greater variability.
For all datasets, the first peak is excluded from the analysis, as it corresponds to an initial transitional phase characterised by high variability and limited reproducibility.\\
	
\noindent\textbf{Idealised data representation.} The data are time series representing the evolution of the bio-markers of the luciferin (indicating when the {Per2} gene is on).
	In order to describe the data at our disposal, we denote :
	\begin{itemize}
		\item $M \geq 1$ the number of luminescence experiments,
		\item $n_j\gg 1$ the number of measurements for the experiment indexed by $j$ where $j=0, \hdots , M-1,$
		\item $s\in \{ WT, Cry2KO \}$ the indication on the type of spheroid used in the experiment, $WT$ (resp. $Cry2KO$) for  spheroids composed of wild type cells (resp. cells with the gene {Cry2} knocked out),
		\item $dt^j$ the time step between observations of experiment indexed by $j = 0, \hdots , M-1.$ 
	\end{itemize}
	For some (large) $n_j \geq 1$, we have measurements $y_i^{n_j}[s]$ of a noisy signal of type $s$ localised around $i/n_j$, so that $i$ is a location parameter and $n_j$ a frequency parameter. We may idealise our data via a representation of the form
	\begin{equation} \label{eq:data_rep_gauss}
		y_i^{n_j}[s]=x^{j}\left(i\times dt^j; s\right)+\sigma^j[s] \xi_i^{n_j},\qquad i= 0,\hdots,n_j-1,\; j= 0,\hdots,M-1,
	\end{equation}
	\noindent
	where 
	$x^{j}(\cdot ; s)$ is the true (unknown) signal of interest and the $\xi_i$ are independent and identically distributed noise measurements, that we assume here to be standard Gaussian.
	This assumption is based on the following : the noise in the light intensity is generated as the superposition of a large number of random independent interactions at the molecular level.
	The quantity $\sigma^j [s] >0$ is a (fixed) noise level characterising the experiment indexed by $j$ of type $s$. 
	\\
	For the sake of clarity, we use a simplified notation in the following to describe a standard set of data which is rigorously formalised by \eqref{eq:data_rep_gauss}:
	\begin{equation} \label{eq:data_rep_gauss_s}
		y_i=x_i+\sigma \xi_i,\qquad i= 0,\hdots,N-1,
	\end{equation}
	where $i$ is the time localisation index and $N\in\mathds{N}^*$ the number of samples.
	Hence, $(x_i)_\seqint$ is the true (unknown) signal of interest. 
	
\subsection{Quantitative features of circadian oscillations}\label{sec:denoise_detrend}
The analysis of time-series data from circadian rhythm experiments conducted on hepatocyte spheroids presents several challenges.
First, the recording time is short relative to the oscillation period; due to experimental constraints, only a maximum of 10 cycles can be observed.
Furthermore, the signal intensity decreases over time, making the signal characteristics less clear.
Additionally, since the oscillation periods are around 24 hours, the relevant range of frequencies is very low and the distinction between trend and oscillatory features is blurry.
Finally, edge effects influence the signal characteristics at the beginning or end of the recording, causing unrealistic phase shifts or amplitude variations at these edges in the spectral domain.
In order to tackle these bottlenecks, we introduce a signal processing method based on wavelet decomposition to extract reliable quantitative features characterising these signals. The following signal processing procedure is using Python language and is available at \url{https://github.com/AnastasiaMARECHAL/InSyncPy}.

\subsubsection{Oscillatory features characterisation}
Continuous wavelet transform (CWT) approaches \cite{torrence1998practical,lilly2017element}, have emerged as powerful analytical tools thanks to their ability to bypass the strict stationarity assumptions required by traditional Fourier analysis, thereby enabling the accurate extraction of instantaneous oscillatory parameters. 
Despite these significant advantages, standard wavelet methods are fundamentally constrained by the uncertainty principle, which imposes an unavoidable trade-off between time and frequency localisation. 
In order to extract quantitative features of oscillating biological signals, existing tools like pyBOAT \cite{schmal2022analysis} typically rely on maximum power ridge tracking using standard Morlet wavelet. 
In the following, we describe a signal processing method based on SynchroSqueezed wavelet Transform (SST) \cite{daubechies2009synchrosqueezed} using the generalised Morlet wavelet \cite{martinez2022applications}.
The advantages are twofold: first, the generalised Morlet wavelet introduces an additional parameter that allows for explicit control over the trade-off between time and frequency localisation. 
Second, the SST sharpens frequency localisation and enhances the precision of instantaneous frequency estimation when analysing complex biological signals.

\paragraph{\textbf{Preprocessing.} }
Data preprocessing consists of denoising, trend removal, and edge trimming. 
First, we denoise the signals via wavelet decomposition and fixed-threshold selective reconstruction. 
Next, we remove the low-frequency baseline trend by smoothing the data with splines. 
Finally, we trim the extremities of the recordings to mitigate edge effects and align them to a common peak, thereby enabling robust comparisons.
More details are provided in Appendix \ref{sec:ap_denoise_detrend}.

\paragraph{\textbf{Quantitative descriptors of the oscillations.}}
\label{section:descriptors}
Now that preprocessing is complete, we start extracting the characteristics of the oscillations.
The main idea is to characterise the damping of the oscillations 
and to quantify the period evolution over time using the continuous wavelet transform (CWT) on a denoised and detrended signal. 

Signals of circadian rhythms are slow-varying oscillations with damped amplitude. 
They can be modeled as a monocomponent signal 
\begin{equation}
x(t) = A(t)e^{i\phi(t)}, \quad A(t) = e^{-\varepsilon t}, \quad \varepsilon>0,
\label{eq:def_quant_amplitude}
\end{equation}
with slowly varying phase $\phi(t)$. 
Quantitative descriptors of oscillatory features are obtained via the continuous wavelet transform (CWT) \cite{daubechies1992ten}: 
\begin{equation}
W_x(a,b) = \frac{1}{\sqrt{a}} \int_{-\infty}^{\infty} x(t) \overline{\psi\Big(\frac{t-b}{a}\Big)} dt,
\label{eq:wlt_coeff}
\end{equation}
where $a$ is the scale (inversely related to frequency), $b$ the time localisation and $\psi$ the appropriate chosen wavelet (in our case the generalised Morlet wavelet).
Then, we use the synchrosqueezed wavelet transform (SST) \cite{daubechies2009synchrosqueezed} to improve the time-frequency resolution.
First, we compute the instantaneous frequency of the signal denoted by $\omega_x$:
$$\omega_x(a,b) = -i\frac{\partial_b W_x(a,b)}{W_x(a,b)} $$
and then we get the SST of the signal:
\begin{equation}
T_x(\omega,b) = \int a^{-\tfrac{3}{2}}W_x(a,b)\delta(\omega_x(a,b) - \omega) da,
\label{eq:sst_coeff}
\end{equation}
where $\delta$ is the dirac function. 
The ridges extraction from the SST allows to track local maxima of the synchrosqueezed energy over time. The quantitative feature characterising the \textbf{frequency evolution} is given by 
\begin{equation}
	\omega_r(b) = \arg\max_\omega |T_x(\omega,b)|,
	\label{eq:def_quant_period}
\end{equation}
and the one characterising the amplitude evolution is 
$$A_{\rm inst}(b) = |T_x(\omega_r(b), b)|.$$
The parameter of the \textbf{amplitude decay}, $\varepsilon$, is estimated by fitting the instantaneous amplitude obtained with the CWT with a model $t\mapsto \exp(-\varepsilon t).$

The SST provides a robust framework for oscillatory features estimation via ridge extraction \cite{iatsenko2016extraction}, ensuring reliable performance under appropriate conditions such as reasonable signal-to-noise ratio, slow frequency variation, and suitable wavelet choice \cite{lilly2017element}.
More details are provided in the Appendix \ref{section:ap_descriptors}.

\paragraph{\textbf{Limitations at low amplitude.}  }
It should be noted that when the instantaneous amplitude of the signal is very low, the ridge extraction becomes more sensitive to noise. 
As a result, the estimated instantaneous period may become less reliable for these low-amplitude segments. 
This limitation is inherent to all time-frequency based approaches.

\subsubsection{Oscillatory features of the {Per2}:luciferase reporter}\label{sec:ana_data}
The luminescence recordings for the WT and Cry2KO cells (cf. Figure \ref{fig:data_brut}) are analysed using the signal processing method described in the previous section.
For the sake of visualisation, we propose the following transformation of the extracted frequency features derived from both the CWT and the SST. 
The frequency estimate has been converted into a period estimate by taking the inverse, and the mean has been computed, resulting in a single period estimator for each luminescence recording.
    \begin{figure}[h!]
		\begin{center}
			\includegraphics[height= 7cm,width=0.7\textwidth]{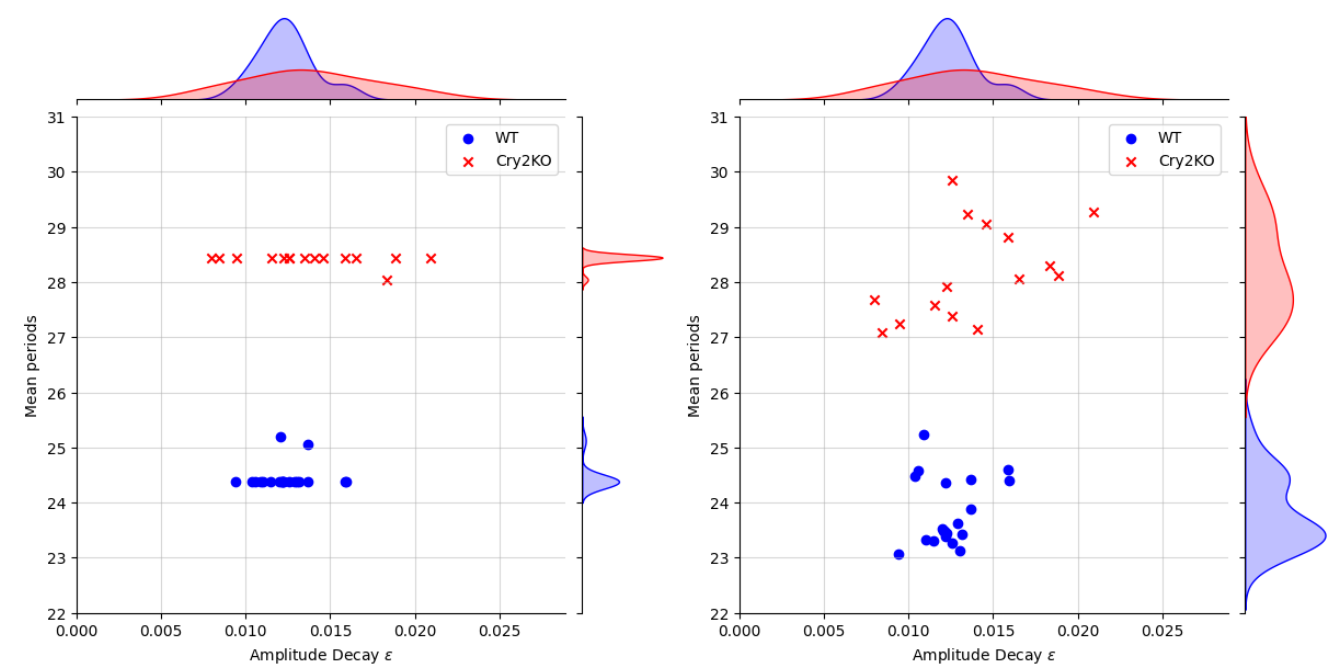}
			\caption{\small Scatter plots of the amplitude decay parameter $\varepsilon$ (x-axis) versus the mean period parameter (y-axis).
The period and amplitude decay estimates are either computed from the CWT coefficients (left) or the SST coefficients (right). 
The marginal distributions above (resp. on the right) the scatter plot characterise the spreading of the amplitude decay parameter's distribution (resp. the period parameter's distribution).
   The color blue (resp. red) is for the WT population experiments (resp. the Cry2KO population experiments).
} 
			\label{fig:scatter_data}
		\end{center}
	\end{figure}
	
\noindent Figure \ref{fig:scatter_data} allows to visualise the amplitude decay parameter against the mean period parameter for each luminescence recording, contrasting the CWT estimates (left panel) with the SST estimates (right panel).
The first striking observation is that there is a clear separation between features of WT cells (blue dot) and ones of Cry2KO cells (red cross). 
This separation is predominantly driven by the period estimators. 
In the CWT analysis, period distributions are highly concentrated around a single value on the y-axis for both cell types.
When examining the SST scatter plot, the period estimates exhibit a wider vertical dispersion. 
However, the fundamental period separation between the two cell populations remains clearly identifiable.
The Cry2KO cells have a longer period (mean: 28.42h) as compared to the WT cells (mean: 24.46h), which is expected from the literature \cite{guenthner2014circadian}.
\\
Concerning the amplitude decay of the oscillations, across both methods, the amplitude parameters' distribution for the WT cells is mainly contained between $0.009$ and $0.016$ whereas the amplitude parameters' distribution for the Cry2KO cells is noticeably more spread (between $0.007$ and $0.021$).
This observation consistently corroborates the commonly accepted fact by biological experimentalists that genetic perturbation on a living model induces more variability for the behaviour of the modified living organism. 
The second striking observation from this analysis is that the variability predominantly impacts the amplitude of the oscillations.

\section{Outputs system to simulate circadian rhythm}

To compare system \eqref{eq:adaptive_kura} to experimental {Per2}::luciferase recordings (see Figure \ref{fig:data_brut}), we adapt the reporter model introduced in \cite{kaji2023enhanced}.
The time evolution of the luminescence produced by the network of oscillators is assimilated to the dynamic of the {Per2} reporter.
In \cite{kaji2023enhanced}, the output system is modelled by a differential equation where the reporter protein's concentration changes based on a clock-regulated production rate, a basal production term, a linear degradation rate, and independent stochastic noise. 
Since the signal processing step removes noise, trend and normalises the amplitude of the oscillations (cf. Section \ref{sec:denoise_detrend}), we model the reporter concentration dynamic the following way. 
We denote $P(t)$ the \textit{ideal} reporter's concentration (without taking into account the impact of the experimental setting) at time $t$ and model its dynamic by:
\begin{equation}
\left\{
\begin{array}{rl}
\frac{d}{dt}P(t)&= \underbrace{\frac{1}{N}\sum_{k = 1}^N f(\theta_k(t))}_{\text{Production driven by the oscillators' phases}} -\underbrace{k_PP(t)}_{\text{Degradation of the luminescence}}\\
P(0) &= P^0,
\end{array}
\right.
\label{eq:per_model}
\end{equation}
where the phases $(\theta_k)_{k=1,\hdots,N}$ are solution of the system \eqref{eq:adaptive_kura} and $k_P>0$ is the degradation rate of the circadian clock reporter.
The reporter production is then modelled through a periodic production function $f$:
\begin{equation}
f(\theta) = 1+\cos\left((\theta-\theta_{PER})\frac{2\pi}{L}\right),
\label{eq:prod_reporter}
\end{equation}
with $\theta_{PER}$ the parameter corresponding to the CT with maximum production of the reporter.
In the mammalian circadian clock architecture described by Takahashi et al. \cite{takahashi2017transcriptional}, {Per2} repressor occupancy peaks at approximately  CT16$\approx 2L/3$.
Hence, we set $\theta_{PER} =2L/3$ in \eqref{eq:prod_reporter} to align the production peak with the biological knowledge.\\
Then, we denote $\tilde{P}(t)$ the centered and renormalised reporter output (where we subtract its mean and divide by its absolute maximum value). 
Moreover, in these experiments, the amplitude of the {Per2}::luciferase signal provides a direct indicator of synchronisation at the population scale. 
When individual cells oscillate with similar phases, their {Per2}'s reporter outputs add up and produce a signal of bigger amplitude compared to a desynchronised state. 
However, since in long-term recordings without medium replacement hepatocytes gradually lose their nutrients, it leads inevitably to cell death which consequently reduces reporter protein levels over time.
Hence, the overall decay of the oscillations amplitude over the whole time window of the experiments (cf. Figure \ref{fig:data_brut}) are modeled by an exponential loss term with parameter $\lambda >0$ and a more realistic output $Y(t)$ is given by: 
\begin{equation}
			Y(t) = \tilde{P}(t)e^{-\lambda t},\quad t>0.
			\label{eq:luci}
\end{equation}
Finally, to align experimental recordings and model output, we set the initial condition of equation \eqref{eq:prod_reporter} such that $P^0$ is a maximal point of the trajectory of the output model:
\begin{equation}
		P^0=\frac{\frac{1}{N}\sum_{j=1}^Nf\left(\theta_j^0\right)}{k_P}.
		\label{eq:P_0}
\end{equation}
Indeed, the experimental data that have been denoised, detrended, centered, normalised and trimmed such that the beginning of the signal corresponds to the first peak of an oscillation. 
This reporter model allows to couple the phase-based description for the cells' circadian rhythms to a protein induced luminescence processes, thus facilitating the comparison between the outputs of the Kuramoto model with experimental data.
It should be noted that luciferase production could be more detailed at the intracellular scale (see e.g. \cite{kurosawa2002comparative}), however since we are interested in the macroscopic behaviour of the cell population adding more complexity (in terms of parameters, size of the system, etc) in the high-dimensional system studied here would be adversarial to that goal.
In this study, we will consider the simplified version proposed in \cite{kurosawa2002comparative} and detailed above.

\begin{figure}[h!]
\begin{center}
\includegraphics[width=0.8\textwidth]{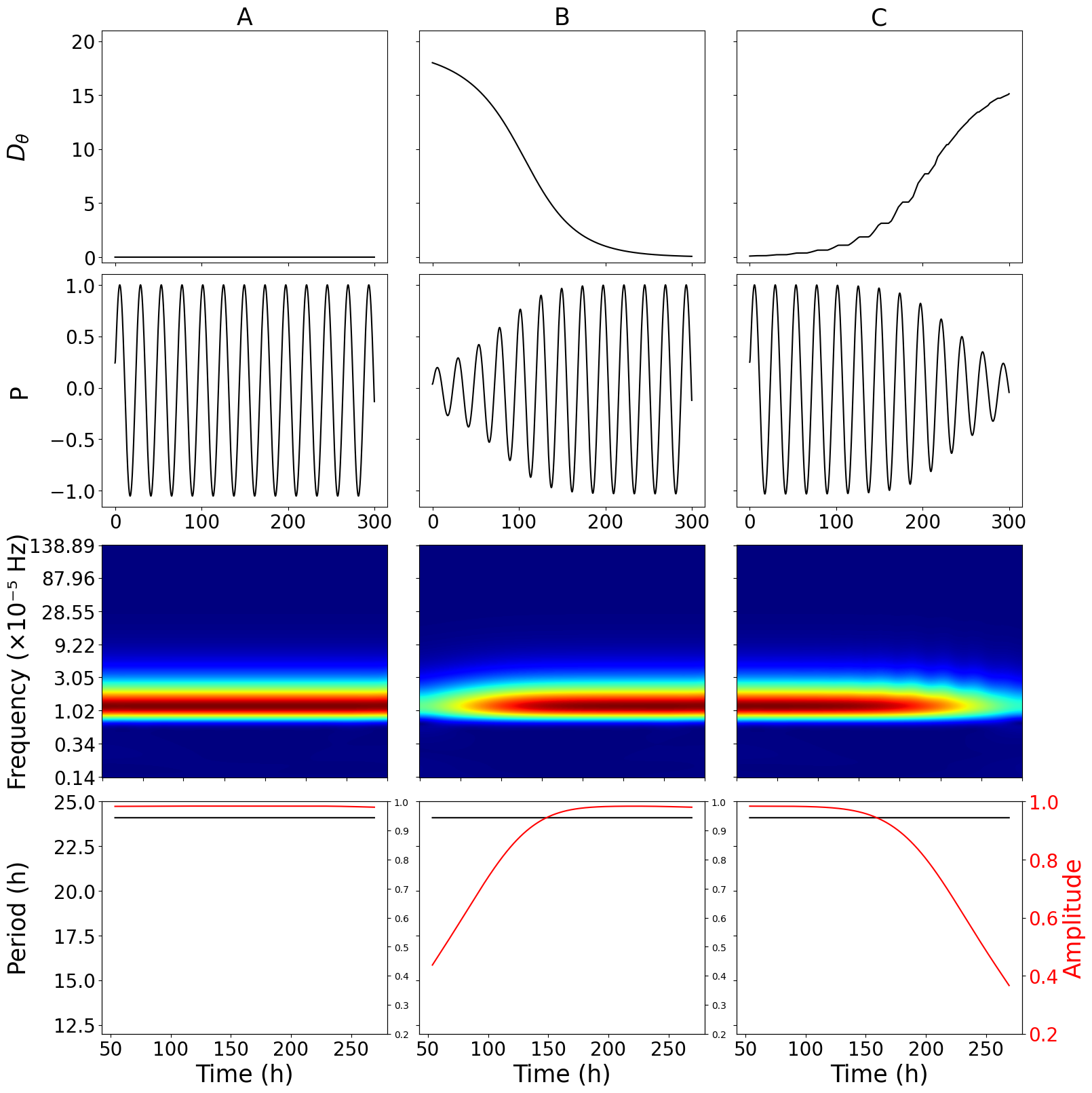}
\end{center}
\caption{\small \textbf{Graphical illustration of the output system.} 
The figure shows the temporal evolution of the phase diameter (first row), the output system \eqref{eq:prod_reporter} (second row), the scaleogram of the output system (third row) and the instantaneous amplitude (red) and instantaneous period (black, last row). 
The x-axes are shared. The simulations were performed with $L = 24$, $k_P = 0.2$, and $P^0$ given in Eq \eqref{eq:P_0}. 
Three different behaviours of the system are shown column-wise: 
\textbf{A:} the oscillators are initially aligned on the same phase and no interaction happens in the system. The parameters values are $N = 10, K = 0, \Delta_\theta = L, D^0 = 0$.
\textbf{B:} The initial phase distribution is spread out and the oscillators are synchronising over time. The parameters values: $N = 10, K = 0.1, \Delta_\theta = L, D^0 = 3L/4$. 
\textbf{C:} The oscillators are aligned on the same phase and the oscillators are desynchronising over time. The parameters values: $N = 2, K = -0.17, \Delta_\theta = L/2, D^0 = 0.1$.}
		\label{fig:output_model}
	\end{figure}

Figure \ref{fig:output_model} illustrates the synchronisation dynamics within a network of circadian oscillators. 
While the phase diameter serves as the rigorous mathematical metric for quantifying network synchronisation, direct experimental observation of this state is challenging, as the available data is limited to macroscopic {Per2}::luciferase bioluminescence signals. 
To highlight the links between the theoretical phase metric (the phase diameter \eqref{eq:phase_diameter}) and observable reporter data, we conducted \textit{in silico} simulations to characterise the reporter output system under three distinct network states: fully synchronised (Figure \ref{fig:output_model} \textbf{A}), dynamically synchronising (Figure \ref{fig:output_model} \textbf{B}), and desynchronising (Figure \ref{fig:output_model} \textbf{C}).

A critical distinction in how synchronisation dynamics manifest in the macroscopic signal can be concluded from this analysis. 
Specifically, the processes of synchronisation and desynchronisation exert decoupled effects on the instantaneous amplitude and the instantaneous period.
The impact of cellular interactions on the instantaneous period is negligible; instead, the instantaneous amplitude emerges as the primary proxy for the underlying phase distribution of the oscillators. 
During network synchronisation, the aggregate output yields robust interference, resulting in high-amplitude oscillations and a concurrent amplification of signal energy visible on the scaleogram (Figure \ref{fig:output_model} \textbf{B}). 
Conversely, as the oscillators desynchronise, phase dispersion causes the signal energy to dissipate, yielding a low-amplitude profile that asymptotically approaches zero (Figure \ref{fig:output_model} \textbf{C}). 
Ultimately, it demonstrates how the temporal variations in signal amplitude and scaleogram energy provide a robust framework for inferring the unobservable phase distributions of cellular circadian clocks.

\section{Calibration procedure}

A robust calibration approach is required to link the mechanistic framework of the adaptive Kuramoto model \eqref{eq:adaptive_kura} and the reporter system \eqref{eq:per_model}-\eqref{eq:luci} with the experimental observations. 
In this section, we detail the methodology used to estimate the model’s parameters. 
We first outline the parameter space, the objective functions, and the filtering algorithm. 
Subsequently, we provide a biological interpretation of the resulting parameter distributions and the correlation matrix, with a specific focus on the implications of parameter identifiability.

\subsection{Parameter space and filtering procedure}\label{sec:param_filter}
\paragraph{The parameter space.} The mathematical model of coupled circadian clocks relies on a set of parameters governing the circadian oscillation, the dynamic coupling of hepatocytes and the dynamics of the luminescence reporter.
The parameters ranges and biological meaning are summarised in Table \ref{tab:para} and additional details are provided in the Appendix \ref{sec:ap_para}. 
\begin{table}[h!]
		\begin{center}
			\begin{tabular}{|c|c|c|c|c|}
\hline
				&Name & Values/Range & Units& Description\\
\hline
				\small Circadian clocks &&&&\\
				 &$L^{\rm WT}$& $24.46$ & h&\small Estimated period of the WT dataset\\
				&$L^{\rm Cry2KO}$& $ 28.42$& h&\small Estimated period of the Cry2KO dataset\\
				&$N$ &100&-&\small Number of oscillators\\
\hline
				\small Initial condition&&&&\\
				& $\sigma$& $[10^{-2};10^2]$& CT &\small Initial distribution spread out parameter \\
				& $P^0$ & see Eq \eqref{eq:P_0} &-& \small Initial photon concentration\\
\hline
				\small Phase model&&&&\\
				& $\Omega$ & $1$& CT/h&\small Intrinsic frequencies\\
				& $K$ & $(0;1)$&-&\small Coupling strength\\
				& $\Delta_\theta$ & $[0;L]$ &CT&\small Width of the coupling interval $I_\theta$ \\
\hline
				\small Output model&&&&\\
				 & $k_{P}$ &$[0;1]$ &h$^{-1}$&\small Degradation rate of the reporter protein \\
				& $\lambda$&$[0;0.1]$&h$^{-1}$&\small Degradation rate of the luminescence signal\\
				\hline
				
			\end{tabular}
			\caption{Table of the parameters of the models \eqref{eq:adaptive_kura} and \eqref{eq:per_model}-\eqref{eq:luci}.}
			\label{tab:bio}
		\end{center}
	\end{table}
Some parameters are fixed a priori based on established properties of mammalian circadian clocks and for computational convenience:
\begin{itemize}
\item \textbf{The period $L$.} Since the instantaneous periods of each experiments show little variability across datasets (cf. Section \ref{sec:ana_data} and Figure \ref{fig:scatter_data}), we fix the period to the mean value, i.e. the WT population exhibits a period of $L^{\rm WT}=24.46$ hours, whereas the Cry2KO population shows a longer period of $L^{\rm Cry2KO}=28.42$ hours.
\item \textbf{The intrinsic frequencies $\Omega$.} Since we are modelling an homogeneous population of hepatocytes with the same period for their circadian clocks, we fix $\Omega = 1$ so that the period of the circadian cycle is only determined by the parameter $L$.
\item \textbf{The initial condition of the reporter system $P^0$.} We set the initial condition for the reporter model as a function of the initial phase distribution in order to fit the experimental data (cf. \eqref{eq:P_0}).
\end{itemize}
Similarly, parameters dictating the dynamical behaviour of the system were constrained to prevent physically impossible rapid state changes, acknowledging the transcriptional-translational feedback loop delays inherent to circadian timekeeping. 
Hence, to ensure biological plausibility and computationally tractable filtering, it is necessary to define boundaries and assumptions for these free parameters:
\begin{itemize}
\item \textbf{The spread of the initial condition $\sigma$.} The synchronisation state of the population at initial time of the experiment is unknown. Then the initial phase distribution is parametrised to account for different levels of synchronisation and we assume it follows a truncated Gaussian distribution on the interval $[0,L]$, centered at $L/2$ with density $f$ such that:
	$$
	f(\phi)d\phi=\frac{\exp\left(-\frac{(\phi-L/2)^2}{2\sigma^2}\right)}
	{\int_0^L \exp\left(-\frac{(x-L/2)^2}{2\sigma^2}\right)\,dx} d\phi,
	\qquad \phi\in[0;L],
	$$
where parameter $\sigma\in[10^{-2},10^{2}]$ controls the spread of the distribution (see Appendix \ref{sec:ap_para}). 
\item \textbf{The coupling strength $K$.} The coupling strength of interactions between individual clocks is assumed in $(0; 1)$ (see Hyp. \ref{hyp:K}). It allows to avoid the totally uncoupled state and the opposite state where $K$ is at its maximum admissible coupling strength.
\item \textbf{The width of the coupling window $\Delta_\theta$.} We set the coupling window width such that $\Delta_\theta \in [0,L]$. It follows that the network of oscillators can interact in the whole spectrum from totally uncoupled to coupled at all stages of the circadian cycle.
\item \textbf{The degradation rate of the reporter's protein $k_P$.} It is assumed that $ k_P\in [0,1]$. 
\item \textbf{The exponential decay of the output signal $\lambda$.} Since the amplitude decay estimates exhibit variability across datasets as explained in Section \ref{sec:ana_data}, the parameter $\lambda$ was treated as a calibration parameter and explored over the range $[0;0.1]$.
\end{itemize}

\paragraph{The objective functions.}  
The calibration of these parameters is framed as an inverse problem. 
We define our objective functions $F_i$ for $i=1,2$ such that $F_1$ is  the mean squared error between the simulated reporter system \eqref{eq:luci} and the averaged sum of the experiments of a specific type of spheroid $s\in\{WT,Cry2KO\}$ (cf. Section \ref{sec:bio_data}) and $F_2$ is the mean squared error between the projections of these two signals in the wavelet domain on a specific scale window.

Let $\Gamma:=\{\sigma, K, \Delta_\theta, k_P, \lambda\}$ denote the set of free parameters and $Y(t;\Gamma)$ denote the numerical approximation of the reporter output \eqref{eq:luci} evaluated at time $t$ with parameters $\Gamma$.
Recall the notations of the data representation (Section \ref{sec:bio_data}) and let $M[s]$ for $s\in\{WT,Cry2KO\}$ be the number of experiments of type $s$ such that $M[WT] + M[Cry2KO]= M$.
We performed the calibration on the two different datasets WT and Cry2KO, aiming to estimate the five parameters: $\sigma, K, \Delta_\theta, k_P$ and $\lambda$, and to compare the results between the two calibrations.
	For each dataset, let $T_j = (n_j-1)\times dt^j$ be the maximal time of the $j$-th recording with $j\in\{1, \hdots, M[s]\}$.	
	
To compare the mathematical model to the experimental data, two metrics are used. 
The first one is the standard mean squared error (MSE), defined as 
	\begin{equation}
	    F_1(\Gamma) = \frac{1}{M[s]}\sum_{j = 1}^{M[s]}\frac{1}{T_j}\int_0^{T_j}(y_j(t) - Y(t;\Gamma))^2dt
	    \label{eq:calib_F1}
	\end{equation}
where $y_j$ is the experimental observation (cf. \eqref{eq:data_rep_gauss_s}).
	This choice of metric allows for a direct comparison between experimental and simulated trajectories over time by measuring the $\mathcal{L}^2$-deviation between the curves. 
	In particular, the MSE penalises divergence in both amplitude and temporal dynamics, making it a particularly suitable indicator for evaluating the model’s ability to reproduce the full shape of the observed signals.
	
	The second metric captures differences in the time–frequency structure of the signals by comparing their wavelet coefficients defining in Eq. \eqref{eq:wlt_coeff}:
	\begin{equation}
	    F_2(\Gamma) = \frac{1}{M[s]}\sum_{j = 1}^{M[s]}\left|\left|\ W_{y_j} - W_{Y(\cdot;\Gamma)}\ \right|\right|_{F}
	    \label{eq:calib_F2}
	\end{equation}
where $||.||_{F}$ is the Frobenius norm given by
$$||W_x||_{F} = \left(\sum_{a, b}|W_x(a, b)|^2\right)^{1/2}.$$
The wavelet coefficients are computed over a restricted range of scales corresponding to relevant periods, namely between 15h and 35h, in order to focus the comparison on a relevant range of scales for the circadian dynamics.

\begin{figure}[!h]
\begin{center}
\begin{tikzpicture}[
    boxblue/.style={
        draw=blue, 
        fill=blue!20, 
        rounded corners=6pt, 
        line width=1pt, 
        minimum width=3cm, 
        minimum height=1.2cm, 
        font=\Large
    },
    boxred/.style={
        draw=red, 
        fill=red!20, 
        rounded corners=6pt, 
        line width=1pt, 
        minimum width=3cm, 
        minimum height=1.2cm, 
        font=\Large
    },
    arrow/.style={
        -stealth, 
        line width=1pt
    }
]

\node (US) at (0, 2.5) {\large Uniform sampling};
\node (GS) at (7.25, 2.5) {\large Gaussian sampling};
\node at (3, -2) {\large Filtering};

\node[boxblue] (J0) at (0, 0) {$\mathcal{J}_0$};
\node[boxblue] (J1) at (0, -4) {$\mathcal{J}_1$};

\node[boxred] (J2) at (6, 0) {$\mathcal{J}_2$};
\node[boxred] (J3) at (6, -4) {$\mathcal{J}_3$};

\draw[arrow] (US) -- (J0);
\draw[arrow] (J0) -- (J1);
\draw[arrow] (J2) -- (J3);

\draw[arrow] (J1.south) 
    -- (0, -5.5) 
    -- node[below=4pt] {\large Mean and covariance} (8.5, -5.5) 
    -- (8.5, 1.5) 
    -- (6, 1.5) 
    -- (J2.north);

\end{tikzpicture}
\end{center}
\caption{Schematic view of the step-wise algorithm used to select parameter sets. 
The collection $\mathcal{J}_0$ of parameter sets is distributed uniformly over the hypercube of Table \ref{tab:bio}. 
The collection $\mathcal{J}_1$ is obtained by filtering $\mathcal{J}_0$ keeping the $0.1\%$ parameters sets with the lowest evaluation of the objective function $F_1$ \eqref{eq:calib_F1}. 
The collection $\mathcal{J}_2$ is obtained by re-sampling the hypercube with a Gaussian distribution fitted on $\mathcal{J}_1$.
The collection $\mathcal{J}_3$ is obtained by filtering $\mathcal{J}_2$ keeping the $0.1\%$ parameters sets with the lowest evaluation of the objective function $F_2$ \eqref{eq:calib_F2}. 
The final collection $\mathcal{J}_3$ is a collection of parameter sets that are consistent with data and minimizes the mean square deviations for the time series and for their projection on the time-frequency domain.}\label{fig:monte_carlo}
\end{figure}
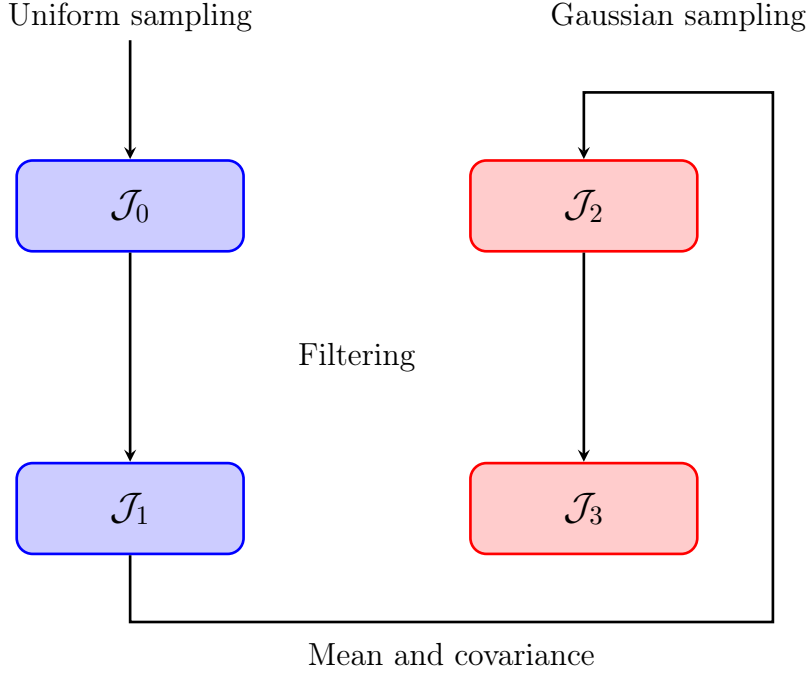

\paragraph{The filtering procedure.} In the context of inverse problem, e.g. parameter estimations, we propose a two-step quasi-Monte-Carlo calibration procedure \cite{chaaya2024continuous} that sequentially captures macroscopic energy dynamics and scale-localised transient features (see Figure \ref{fig:monte_carlo}). 
For a given set of parameters $\Gamma$, the computation of both criterion $F_1$ and $F_2$ requires two numerical simulation approximations, one for the Wild type (WT) data and one for the Cry2KO data. 
The integrals over time intervals are computed with a trapezoidal rule. 
The computation of the gradient $\nabla_\Gamma F_i(\Gamma)$ of any coordinate of the criterion can not be obtained analytically. 
This means that we have no descent direction to minimize our criteria.
The space we have to explore is the hypercube of dimension 5 defined by the range column of Table \ref{tab:bio}. 
Hence, we rely on a two-stage algorithm.
In the first step, the standard MSE $F_1$ \eqref{eq:calib_F1} between the empirical data and model output is used to explore uniformly the whole range of parameters and fit the overall amplitude and frequency dynamics of the system. 
This initial time-domain optimization establishes a preliminary basin of attraction in the parameter space.
From this collection, we adjust a Gaussian distribution truncated to the initial hypercube.
In a second step, the distribution is used to draw new parameter sets that are informed by the data.
We refine the parameter estimates by selecting among these draws the best parameter sets according to the second criterion using the Frobenius norm of the error matrix $F_2$ \eqref{eq:calib_F2} between the CWT of the empirical and simulated data, strictly confined to a predefined and pertinent scale range (the range with the most biologically relevance for circadian clock data). 
This scale-restricted CWT mapping acts as a targeted time-frequency filter rejecting out-of-band frequency features.
The results of the calibration procedure are summarized in the Table \ref{tab:para} and Figure \ref{fig:result} shows the model outputs corresponding to the best $0.1\%$ parameter sets obtained during the calibration procedure.
\begin{table}[h!]
\begin{center}
	\begin{tabular}{|c|c|c|c|c|c|}
			\hline
			& $\sigma$ & $\Delta_\theta$ & $K$ & $k_P$ & $\lambda$\\
			\hline
			Range & $[10^{-2};10^2]$ & $[0; L]$ & $[0; 1]$ & $[0; 1]$ & $[0; 0.1]$\\
			\hline
			WT& $18.96\pm 2.79$ & $0.69\pm 0.16$ & $2.20\pm 0.77$ & $0.21\pm 0.018$ & $0.015 \pm 0.0005$\\
			Cry2KO & $26.40\pm 1.29$ & $0.87\pm 0.08$ & $6.78\pm 1.49$ & $0.15\pm 0.046$ & $0.019\pm 0.0017$\\
			\hline
	\end{tabular}
\end{center}
	\caption{Outputs of the filtering procedure. Range for the free parameters: estimated mean values of the parameters with standard deviation for the two data types: Wild type (WT), cells with {Cry2} gene knocked out (Cry2KO).}
		\label{tab:para}
	\end{table}

\subsection{Interpretation of the calibration results: posterior analysis and phenotypic characterisation}

Rather than reducing complex dynamical models to a deterministic point estimate, we try to understand the parameter uncertainty. 
In this non-linear system, a multitude of distinct parameter combinations can yield output errors functionally indistinguishable from a visual perspective based on the biological variability (see Figure \ref{fig:result}). 
Consequently, relying on an isolated "optimal" parameter set invariably forces the calibration to overfit to the mean trajectory and creates an artificial sense of certainty.
Furthermore, retaining an ensemble of parameter sets explicitly acknowledges the inherent physical heterogeneity of the biological process and provides a more accurate representation of system reliability.
To assess the identifiability of the mathematical model and the physiological impact of the Cry2 knockout, the posterior parameter distributions, defined as the top 0.1$\%$ of tuples minimizing the objective functions, were subjected to a comprehensive analysis. \\
Recall that the observed system behavior is governed by two interacting modules.
The oscillators' network is parametrised by its coupling strength ($K$) and its coupling window size ($\Delta_\theta$), which dictate the phase interaction and synchronization capacity of the cellular population. 
The initial state of this network is characterized by $\sigma$, representing the spread of the initial distribution of the oscillators. 
Concurrently, the observable luminescence output is modulated by the intrinsic reporter degradation rate ($k_p$) and a macroscopic signal decay parameter ($\lambda$), which captures the exponential attenuation of the signal amplitude due to experimental constraints and progressive cell death.\\
First, we analyse the marginal probability densities (see Figure \ref{fig:marginal_distributions}) in order to characterise the phenotypic effect of the knockout of {Cry2} on the parameters of the mechanistic model. 
Then, we assess the monotonic dependencies and compensatory mechanisms via Spearman rank correlation (Figure \ref{fig:spearman_correlation}).

\begin{figure}[h!]
\centering
\includegraphics[width=\textwidth]{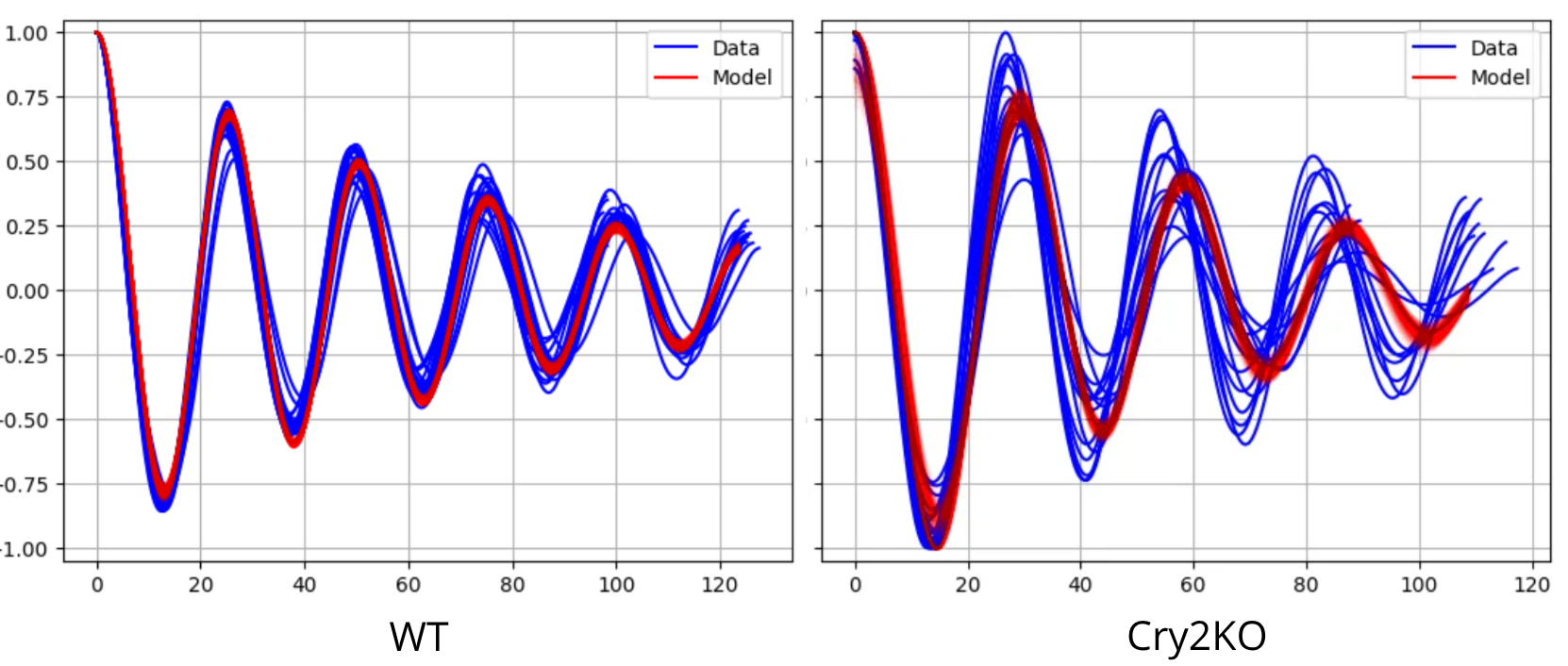}
\caption{Graphs of the processed experimental data (blue) and the calibrated numerical outputs (red). Left: WT dataset, right: Cry2KO dataset. The simulations correspond to the trajectories associated to the parameter sets $\mathcal{J}_3$ representing the $0.1\%$ parameter sets with lowest values for the objective function $F_2$ .}\label{fig:result}
\end{figure}

\paragraph{ Validation and trajectory dispersion} The two-step quasi-Monte Carlo calibration algorithm successfully identified parameter sets that reproduce the main features of the experimental signals across both datasets.
A first assessment of the simulated state space (Figure \ref{fig:result}) reveals that the calibrated trajectories are more tightly concentrated for the Wild Type (WT) dataset than for the Cry2KO dataset. 
This disparity in dispersion accurately reflects the larger intrinsic biological variability observed in the corresponding Cry2KO experimental recordings. 

\begin{figure}[htbp]
    \centering
    \includegraphics[width=0.8\textwidth]{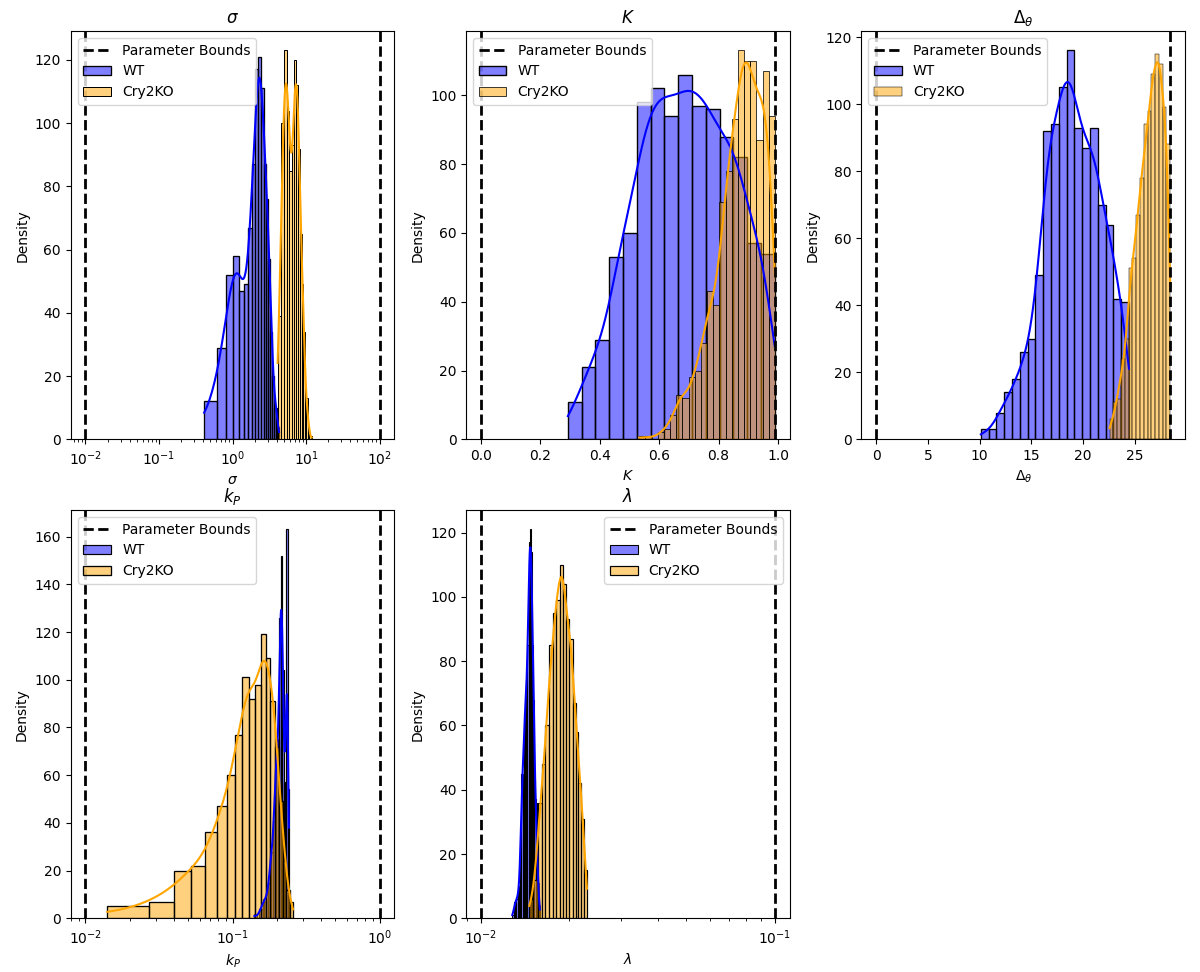}
    \caption{Marginal posterior probability densities for the five model parameters ($\sigma$, $K$, $\Delta_\theta$, $k_p$, and $\lambda$). The blue histograms correspond to WT, while the orange histograms represent Cry2KO. Dashed black vertical lines indicate the uniform prior boundaries established during the initial sampling phase. 
    }
    \label{fig:marginal_distributions}
\end{figure}

\paragraph{Phenotypic divergence and marginal constraints} The marginal distributions of the accepted parameter sets (Figure \ref{fig:marginal_distributions}) capture the phenotypic shift between the WT and Cry2KO settings. 
The posterior densities for the parameters $\sigma$, $\Delta_\theta$, and $\lambda$ exhibit distinct, non-overlapping separations between the two experimental conditions. 
This divergence provides strong quantitative evidence that the knockout of {Cry2} gene fundamentally alters the biophysical mechanisms governed by these specific parameters (e.g. the initial spread of the phase distribution, the length of the phase coupling interval and the exponential decay of the luminescence signal).\\ 
Interestingly, the increased trajectory variability observed in the Cry2KO dataset is not uniformly reflected across all calibrated parameters. 
After calibration, the reporter model parameters ($k_P$ and $\lambda$) consistently display much narrower, more identifiable distributions than the parameters of the adaptative Kuramoto model ($K$ and $\Delta_\theta$) in both datasets. 
Furthermore, the broader dispersion of the Cry2KO system is exclusively absorbed by these reporter parameters; $k_P$ and $\lambda$ exhibit notably larger dispersion for Cry2KO compared to WT, whereas the parameters of the Kuramoto model do not show this same trend of increased variance.
This also suggests a hierarchy between the parameters' identifiability where the parameters of the reporter model are more practically identifiable than the parameters of the Kuramoto model.\\
In addition, the marginal densities also reveal critical nuances regarding parameter constraints. 
For both datasets, the accepted tuples successfully satisfy the physical constraint $\Delta_\theta < L$. 
However, the posterior mass for $\Delta_\theta$ in the Cry2KO condition is heavily truncated against the upper bound of the prior domain, suggesting the true physiological optimum may press closely against this predefined boundary. \\
Furthermore, while the parameter $k_p$ is tightly constrained and identifiable in the WT condition, its posterior density degenerates into a broad, left-skewed distribution in the Cry2KO setting. 
This loss of identifiability implies that the system shifts into a dynamic regime where the precise value of $k_p$ becomes less critical to the targeted physical scales. 
Lastly, the parameter $K$ exhibits practical unidentifiability across both experimental settings; while the data gently preferentially weights higher values, the posterior spans nearly the entire upper half of the prior range suggesting strong coupled interactions between the circadian clocks.

\begin{figure}[htbp]
    \centering
    \includegraphics[width=\textwidth]{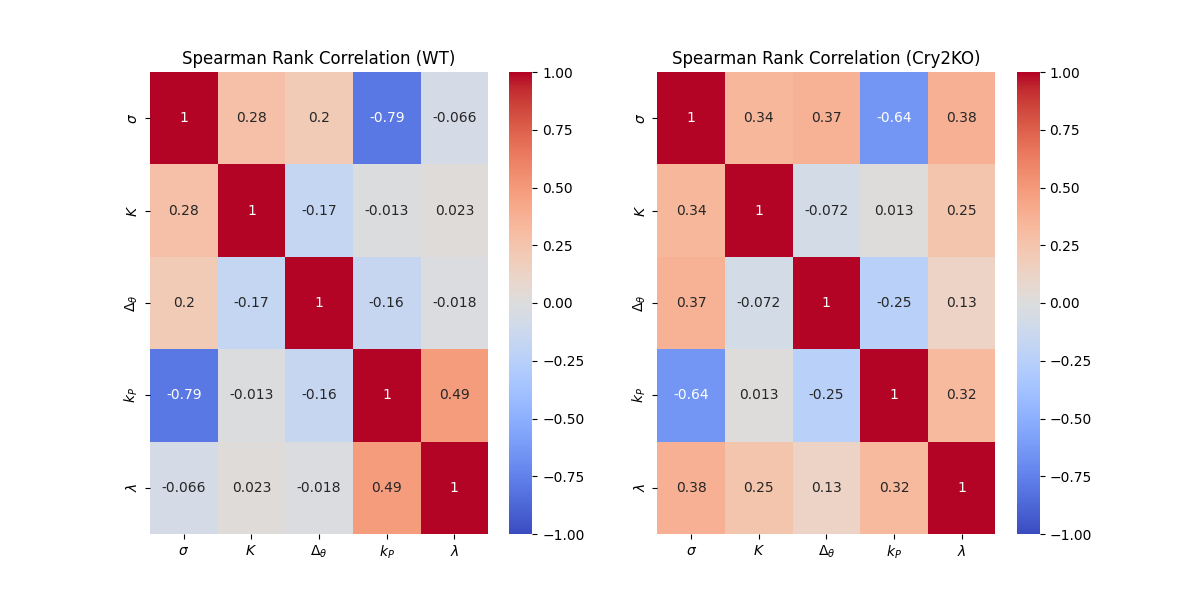}
    \caption{Spearman rank correlation heatmaps of the parameter sets for the WT (left) and Cry2KO (right) conditions. This metric quantifies the strength of monotonic dependencies between parameter pairs. 
    }
    \label{fig:spearman_correlation}
\end{figure}

\paragraph{Parameter identifiability and compensatory mechanisms} By examining the posterior parameter distributions, specifically the Spearman rank correlations in Figure \ref{fig:spearman_correlation}, of WT and Cry2KO experimental settings, we identify compensatory mechanisms between the initial system state, network coupling strength, and reporter degradation kinetics.\\
In both experimental conditions, the most prominent parametric dependency is the strong negative correlation between the initial distribution spread $\sigma$ and the reporter degradation rate $k_p$. 
In the WT setting, this coupling is highly pronounced ($-0.79$) as well as in the Cry2KO setting ($-0.64$).
It confirms a strong trade-off in both experimental settings: an increase in $\sigma$ must be compensated by a decrease in $k_p$ to maintain the same system output.\\ 
The most striking divergence between the two settings occurs in the relationship between the initial distribution spread $\sigma$ and the macroscopic decay rate $\lambda$. 
In the WT model, these parameters are statistically independent ($-0.066$), indicating that the initial state of the network does not interfere with the inference of long-term signal attenuation caused by experimental constraints.
However, in the Cry2KO setting, a moderate positive correlation emerges ($0.38$). 
This shift introduces an identifiability challenge: a highly heterogeneous initial state (high $\sigma$) combined with rapid exponential decay (high $\lambda$) produces a macroscopic output indistinguishable from a more homogeneous initial state with lower macroscopic decay. 
Biologically, this suggests that the Cry2KO mutation diminishes the system's synchronisation robustness over time. The progressive desynchronisation characteristic of the knockout mimics the macroscopic amplitude decay typically attributed to cell death, thereby confounding the estimation of $\lambda$.\\
Furthermore, the knockout condition induces a tighter functional linkage between the initial spread $\sigma$ and the network coupling parameters $\Delta_\theta$ and $K$. 
The positive correlation between $\sigma$ and the coupling window $\Delta_\theta$ nearly doubles from the WT to the Cry2KO condition (from $0.2$ for WT to $0.37$ for Cry2KO), alongside a slight increase in correlation with the coupling strength $K$ (from 0.28 for WT to 0.34 for Cry2KO). 
This indicates that in the Cry2KO network, achieving a specific synchronized state from a highly dispersed initial population requires a proportionally wider coupling window and stronger interaction strength. 
The WT network appears more robust, capable of overcoming initial heterogeneity with less reliance on specific coupling geometries.\\
Ultimately, the transition from the WT to the Cry2KO phenotype is characterised not merely by isolated parameter shifts, but by a global restructuring of the parameter dependency landscape. 
While it slightly relaxes its tightest dependency with $k_P$, it introduces stronger dependencies with the rest of the parameters, particularly with $\lambda$ and to a lesser extent with $\Delta_\theta$.
The identifiability of the long-term viability parameter $\lambda$ is significantly compromised in the knockout model due to newly formed interdependencies with the initial phase spread $\sigma$, with the coupling strength $K$ (from $0.023$ for WT to $0.25$ for Cry2KO) and with the coupling window size $\Delta_\theta$ (from $-0.018$ for WT to $0.13$ for Cry2KO). 
The Cry2KO mutation fundamentally alters the identifiability landscape, decoupling short-term reporter dynamics while introducing novel confounding effects between initial heterogeneity and long-term macroscopic signal decay and between the macroscopic signal decay and the coupling parameters of the oscillators' network.

\section{Perspectives}

	This article introduces a mathematical deterministic model based on the Kuramoto phase model to describes the circadian clocks of a population of cells and their interactions.
	To this end, a phase-dependent coupling interaction kernel is introduced, limiting oscillators' coupling to the cells whose phase of the circadian cycle lies within a specific interval.
	This assumption is motivated by the hypothesis formulated in \cite{finger2021intercellular}, which suggests that cells receive or interpret a chemical signal only during certain stages of the cycle.

A mathematical study of this adaptative Kuramoto model is proposed.
It is shown that, in the case of identical oscillators, the system tends to a complete phase synchronised state under some conditions.
These conditions are the same as those required to achieve complete synchronisation for the seminal Kuramoto model, with the addition of one more condition on the initial phase distribution dispersal and the length of the coupling window. 
An estimate of the speed of convergence toward this synchronised state is achieved.
	
This new mathematical formalism provides a more realistic description of an interacting network of cells' circadian clocks. 
A signal processing approach based on wavelet decomposition is proposed to extract relevant oscillatory features from experimental data and its Python code is available at \url{https://github.com/AnastasiaMARECHAL/InSyncPy}.
The mathematical model is then finely calibrated to the available data using two-step quasi-Monte Carlo algorithm.
Parameters distributions are obtained for two experimental settings: wild type cells and cells with the {Cry2} gene knocked out. 
A comparative study of the parameters dependencies is performed and highlights the phenotypic divergence. 
Though the knockout condition introduces more heterogeneity in the experimental data which reflects for a stronger biological variability, the mathematical model successfully reproduces the trajectory corresponding to the mean behaviour for both experimental conditions.
However, while it is shown that the wild type setting is capable of generating more robust oscillations and that the model can identify and characterise the coupling interactions, the knockout situation alters the identifiability landscape of the models parameters by revealing stronger compensatory phenomena which leads to an ill-posed calibration problem.

Building upon the adaptive Kuramoto framework developed in this study to describe cellular circadian clock interactions, several perspectives emerge. First, an extension of the current model to integrate heterogeneous oscillator networks presents a critical next step. 
The present analysis characterises the synchronization dynamics within isolated, homogenous experimental conditions. 
A natural progression involves the theoretical and experimental integration of these distinct phenotypes, such as co-culturing or simulating mixed populations. 
This approach would be assimilated to the study of non-identical oscillator networks, allowing for the rigorous investigation of how sub-populations with disparate intrinsic parameter distributions, phase dynamics, and coupling capacities interact to shape global entrainment. 
Analysing such mixed networks is essential for understanding the robustness of circadian synchronisation in physiologically complex tissues where diverse cellular phenotypes coexist and communicate.

Second, incorporating explicit spatial topology into the modelling framework would significantly enhance its biological realism. 
The current experimental data and corresponding model formulations neglect the spatial architecture, relying primarily on topological interaction networks without physical coordinates. 
However, spatial proximity inherently dictates the intercellular signalling. 
Introducing a spatial variable would enable the study of emergent, localised synchronisation phenomena such as travelling waves, localised phase clusters across a cellular architecture. 
To computationally and mathematically accommodate this spatial dimension, it would be highly advantageous to transition from the discrete Kuramoto model to its corresponding kinetic formulation. 
Operating in the continuum limit, the kinetic Kuramoto model is far better equipped to handle large-scale, continuous distributions of oscillators over a spatial domain, providing a robust mathematical framework to analyse the continuous spatio-temporal dynamics of cells' circadian clocks.

\insertcreditsstatement

\section*{Declarations}
This research was supported by the French National Agency for Research through project InSync ANR-22-CE45-0012-01 and by LABEX SIGNALIFE (ANR-11-LABX-0028) and the Idex program DYNABIO from Universit\'e C\^ote d'Azur. The authors declares no competing interests.
The authors would like to thank Dr G. Kon Kam King for insightful discussions and fruitful interactions.\\
\textbf{Availability of data and materials.} This article introduces a model of ODEs for formalising the network of circadian oscillators. 
Experimental data are used to inform and validate the mechanistic model.
The experimental methodology used to generate the datasets will be described elswhere (Krawczyk et al, in preparation) and all luminescent recording  data for the WT and Cry2KO experiments are available upon request.  
A signal processing procedure based on wavelet decomposition is proposed to extract quantitative features of the oscillatory dynamics. 
The code is available at \url{https://github.com/AnastasiaMARECHAL/InSyncPy}.
An algorithm based on previous work is then proposed to calibrate the deterministic model and measure the goodness of fit with experimental data.

\bibliographystyle{plain}
\bibliography{ref2.bib}

\appendix
	\section{Details on the Data Processing}
Several issues arise when analysing the time series coming from circadian clock experiments on spheroids of hepatocytes cells :
\begin{enumerate}
	\item the length of the signal is relatively short. 
	The duration of the experiments is limited by the protocol allowing to sustain alive hepatocytes in the spheroids. Since, no medium change is possible during the recording of the light intensity the cells are bound to die by a lack of necessary nutrients. 
	This means that it is only possible to record ten cycles at most.
	Moreover, the signal intensity is generally decreasing over time due to the cells' degradation.
	\item On this data set, the limited number of cycles combined to the averaged duration of 24 hours for the circadian clock makes it challenging to extract information on the frequency domain.
	The frequency of the features of interest needs to be precisely estimated on a frequency range up to $1e^{-4}$ Hz. 	
	Thus, the seminal way to perform a time frequency analysis of the signal using short-time Fourier transform is facing a key limitation in this study. 
	Circadian have long periods ($\approx$24 hours) and hence require long windows to accurately identify oscillatory features, however the fixed (and large) window size is particularly detrimental to the time localization of the features~\cite{sejdic2009time}.
	\item The quantitative descriptors that we want to extract from these signals are also affected by edge effects. 
	In the context of circadian data, which involves low-frequency oscillations, edge effects are especially problematic: meaningful features near the beginning or end of the recording (e.g., phase shifts, amplitude changes) can be misrepresented or attenuated, leading to inaccurate interpretation of biological rhythms. 
\end{enumerate}

Hence, in this nonparametric setting, we aim at detecting and quantitatively characterising from the data $(y_i)_\seqint$ the features displayed by $(x_i)_\seqint$ at a specific frequency range corresponding to experimentally observed circadian periods. 
Therefore, we need to investigate the shape of $(x_i)_\seqint$, which requires smoothing in order to get rid of the noise $(\xi_i)_\seqint$. 
Another difficulty is to separate the circadian clocks' features from the low frequency features emerging from experimental measurements and resulting in trends in the signal.

\subsection{Pre-processing: noise and trend removal, trimming the signal}\label{sec:ap_denoise_detrend}
Denoising and detrending a signal prior to time-frequency analysis enhances the interpretability and accuracy of extracted features by minimising confounding noise and slow drifts that overshadow true oscillatory content.
Denoising removes high-frequency artifacts that can leak down on the frequency axis in the time-frequency domain, while detrending removes low-frequency baseline shifts that may distort the localisation and amplitude of wavelet coefficients.
Thus, this pre-processing step improves spectral resolution, reduces edge artifacts, and facilitates robust detection of biologically relevant dynamics.\\

\paragraph{\textbf{Denoising processing.}} These signals display additive noise with damped oscillations, meaning that although the noise level remains constant through time, the ratio signal/noise is decreasing. 
Hence, the impact of the noise becomes more significant as time passes and affects negatively the quantitative estimators unevenly with respect to time.
The first step of the signal processing is an attempt to remove the white noise influence by filtering in the wavelet domain (c.f. Chap. 11 in \cite{mallat1999wavelet}).
This process corresponds to a low-pass filter that allows to keeps information below a frequency threshold $\omega$.\\ 
The method we use in this step, which we summarise below, was introduced in \cite{donoho1994ideal} and is based on selective wavelet reconstruction.
For the sake of simplicity, suppose that for our data $\left( y_i \right)_\seqint $  described in \eqref{eq:data_rep_gauss_s}, $N= 2^{J+1}$ for $J\in \mathds{N}$ and suppose that the signal of interest $x$ is sufficiently smooth.
Hence, one may construct an orthogonal matrix $\mathcal{W} \in \mathcal{M}_N(\mathds{K})$, denoted the finite wavelet transform matrix (c.f. Chap. 3 in \cite{daubechies1992ten}, Chap. 7 in \cite{mallat1999wavelet}), which is the discrete wavelet transform operator :
\begin{equation}
	\left(w_i\right)_\seqint = \mathcal{W}\left( y_i \right)_\seqint,
	\label{eq:dwt_sig}
\end{equation}
and the following inverse formula holds $$ \left(y_i\right)_\seqint = \mathcal{W}^T\left( w_i \right)_\seqint.$$
Usually, the vector $\left(w_i\right)_\seqint$ is indexed dyadically:
$$w_{j,k} : \quad j = 0,\hdots, J; \quad k = 0, \hdots, 2^j-1 $$
and we rewrite the signal as a sum of basis elements $W_{jk}$ with (wavelet) coefficients $w_{j,k}$:
$$y_i = \sum\limits_{j,k} w_{j,k} W_{jk}(i).$$
For the mother wavelet $\psi$, the basis elements are defined as the following:
$$\sqrt{N} W_{jk}(i) = 2^{j/2}\psi\left(2^j \tfrac{i}{N} -k \right). $$
Hence, the signal projected in the time-scale domain is expressed as the following:
\begin{equation}
	w_{j,k} = \theta_{j,k} + \sigma \xi_{j,k}, 
	\label{eq:sig_wav_basis}
\end{equation}
where $\xi_{j,k}$ are independent and identically distributed standard Gaussian noise. 
It follows from the Parceval formula that for the estimators denoted $\hat{x}$ and $\hat{\theta}$:
$$\mathds{E}\|x-\hat{x}\|^2_2 = \mathds{E}\|\theta-\hat{\theta}\|^2_2. $$ 
Finally, the process of reconstructing the denoised signal $\left(\hat{x}_i\right)_\seqint$ consists of performing the wavelet transform, applying a \textit{keep-or-kill} strategy with a hard thresholding and performing the inverse wavelet transform.
Hence, the estimators $\left(\hat{\theta}_{j,k}\right)_{j,k}$ is given by:
$$\hat{\theta}_{j,k} =w_{j,k}\mathds{1}\left\{w_{j,k} > \sqrt{2\log N}\hat{\sigma} \right\},$$
where $\hat{\sigma}$ is the estimated noise level corresponding to the median absolute deviation of $(w_{J,k})_k$ (the coefficients at the finest scale assumed to be pure noise). 
This assumption is justified in our case since there is a large gap between the frequency of the real signal features (around $1e-4$ Hz) and the frequency sampling ($1.67e-3$ Hz).\\
In conclusion, the reconstructed denoised signal is obtained using the inverse formula:
$$\hat{x} = \mathcal{W}^T \hat{\theta} $$
and, thanks to Theorem 4 in \cite{donoho1994ideal}, the following holds
\begin{equation}
	\mathds{E}\|x -\hat{x}\|^2_2 \leq L_N \left(\frac{\sigma^2}{N} + R_{N,\sigma}(y,x) \right),
\end{equation}  
where $L_N \sim 2\log N$ and 
$$R_{N,\sigma}(y,x) = \inf\limits_\delta \mathds{E} \left\|\sum\limits_{(j,k)\in\delta} w_{j,k}W_{jk} - x \right\|^2_2 .$$
Moreover suppose $x$ is a polynomial of finite order, the squared residuals $ R_{N,\sigma}(y,x) $ satisfies 
$$R_{N,\sigma}(y,x) = O\left(\frac{\sigma^2\log N}{N}\right). $$ 

\paragraph{\textbf{Detrending processing.}} Once the denoised signal $\left(\hat{x}_i\right)_\seqint$ is obtained , our goal is to remove the low frequency features that corresponds to experimental artefacts characterised by slow varying dynamics (changes on a time scale of several days). 
These low frequency features in the time-frequency domain correspond to the trend of the signal.  
An efficient way to remove these features is to minimise the penalised residual sum of squares (c.f. Chap. 5 in \cite{hastie2009elements}):
\begin{equation}\label{eq:smooth_splines}
	\mathcal{F}(\hat\vartheta)(\cdot) = \arg\!\min\limits_{f\in \mathcal{C}^2(\Omega)} \sum\limits_{i=0}^{N-1}\left(f(i\times dt) -\hat{x}_i\right)^2 + \lambda \int_\Omega f^{(2)}(u)du,
\end{equation}
where $\lambda$ is the smoothing parameters controlling the trade-off between the goodness-of-fit to the data and the smoothness of the trend estimators, where $\hat\vartheta\in\mathds{R}^N$ is the vector of estimated parameters which is a linear combination of $ \hat{x}_i$ and where the unique minimiser is a linear combination of natural cubic splines $(\phi_i)_\seqint$. 
Hence  $$\mathcal{F}(\hat\vartheta)(\cdot) = \sum\limits_{i=0}^{N-1}\hat\vartheta_i \phi_i(\cdot)$$
and since the $(\phi_i)_\seqint$ forms a set of basis functions of natural splines, the following holds (c.f. Chap. 5 in \cite{hastie2009elements}) :
$$\left(\hat\vartheta_i\right)_\seqint = S_\lambda \left(\hat{x}_i\right)_\seqint$$
and $ S_\lambda = M\left(M^T M + \lambda \Omega\right)^{-1} M^T$ 
with $M,\Omega \in \mathcal{M}_N(\mathds{R})$ and $M_{ij} = \phi_j(i\times dt) $ and $\Omega_{ij} =\int \phi_j(t)\phi_i(t)dt .$
We denote by $(\tilde{x}_i)_\seqint$ the detrended signal
\begin{equation}
	\tilde{x}_i =\hat{x}_i - \mathcal{F}(\hat\vartheta)(i \times dt), \quad i=0, \hdots, N-1.
	\label{eq:detrend_sig}
\end{equation}
It is important to emphasise two points. 
The first is that the smoothing parameter $\lambda$ plays a significant role in the trend approximation. If $\lambda =0$, the trend is any function interpolating the data and if $\lambda = \infty$, the trend is the best linear fit of the data.
The second point is that poor or no detrending  implies a persistence of edge effects that negatively impact the time-frequency analysis since the signal is non-stationary. \\

\paragraph{\textbf{Trimming the signal.}} The trimming preprocessing step is designed to remove the initial transient oscillation and to align the retained signal segment to zero before performing the time–frequency analysis.
The practical motivations are twofold: 
\begin{itemize}
	\item eliminating the initial oscillation reduces bias from start-up transients that contaminate wavelet estimates through edge effects. Indeed, the dexamethasone pulse used to synchronise the circadian clock also induces a transient, immediate and strong perturbation of the cells that last for few hours and which impact on multiple non circadian clock pathways.
	\item Cutting at a zero crossing reduces phase discontinuities at the trim point, improves the behavior of transforms that assume local stationarity or continuity \cite{torrence1998practical}.
\end{itemize}
The trimming preprocessing step is performed by localising the peaks of the absolute value of the denoised and detrended signal $\left(\tilde{x}_i\right)_\seqint$ with a minimal inter-peak distance set to $5$ hours. 
Then we discard the portion of the signal prior to the second detected peak to remove a full revolution.

\subsection{Quantitative descriptors of the circadian oscillators for a population of cells}
\label{section:ap_descriptors}
The experimental signals processed in this study have damped harmonic oscillations and exhibit a nonstationary behaviour. 
The main idea is to characterise the damping of the oscillations 
and to quantify the period evolution over time using the continuous wavelet transform (CWT) and its Synchrosqueezed Transform (SST) on a denoised and detrended signal. 
In the following, we denote $(\tilde{x}_i)_\seqint$ the signal assimilated to $(x_i)_\seqint$ where the trend features have been removed.\\

\paragraph{\textbf{About the CWT.}} The continuous wavelet transform (CWT) provides a powerful framework for time-frequency analysis of non-stationary signals (cf. \cite{daubechies1992ten,mallat1999wavelet}).
For a signal $x(\cdot) \in L^2(\mathds{R})$ and a mother wavelet $\psi \in L^2(\mathds{R})$ (i.e. a normalised function with an average value equal to $0$), the continuous wavelet transform of $f$ is defined as:
\begin{equation*}
	W_x(a,b) = \frac{1}{\sqrt{a}} \int_{-\infty}^{\infty} x(t) \overline{\psi\left(\frac{t-b}{a}\right)} dt
\end{equation*}
where $\overline{\psi}$ denotes the complex conjugate of $\psi$. The parameter $a > 0$ is the scale parameter (inversely related to the frequency), $b \in \mathds{R}$ is the translation parameter linked to the time localisation.\\

\paragraph{\textbf{About the SST.}}The synchrosqueezed wavelet transform (SST) \cite{daubechies2009synchrosqueezed} is often used to improve the time-frequency resolution involving a rigorous energy reassignment method.
The fundamental mechanism of the SST relies on extracting the instantaneous frequency directly from the CWT. 
For any point $(a, b)$ where the wavelet coefficients are non-vanishing ($W_x(a, b) \neq 0$), the two-dimensional candidate instantaneous frequency $\omega(a, b)$ is computed as the analytical phase derivative with respect to time:
$$\omega(a, b) = -i \frac{1}{W_x(a, b)} \frac{\partial W_x(a, b)}{\partial b}$$
Once computed, the SST reallocates the diffuse oscillatory energy from the time-scale domain $(b, a)$ directly to the time-frequency domain $(b, \omega)$ by "squeezing" it along the frequency axis. 
Formally, the discrete synchrosqueezed transform $T_x(\omega_l, b)$ is defined by mapping the wavelet coefficients to a center frequency $\omega_l$ within a discrete frequency bin $[\omega_l - \Delta\omega/2, \omega_l + \Delta\omega/2]$ of resolution $\Delta\omega$:
$$T_x(\omega_l, b) = \frac{1}{\Delta\omega} \sum_{a_k \in \Omega_l} W_s(a_k, b) a_k^{-3/2} (\Delta a)_k$$
where $a_k$ are the logarithmically spaced discrete scales, $\Omega_l =\{ a: |\omega(a, b) - \omega_l| \le \Delta\omega/2\}$ is the frequency window of length $\Delta \omega$ and center $\omega_l$  and $(\Delta a)_k = a_k - a_{k-1}$ is the scale step. 
Consequently, the original signal can be robustly reconstructed by integrating the synchrosqueezed representation over the frequency axis:
$$x(t) = \text{Re} \left[ \frac{1}{C_\psi} \sum_{l} T_x(\omega_l, t) \Delta\omega \right]$$
where $C_\psi$ is the admissibility constant of the mother wavelet $\psi$.

\paragraph{\textbf{Quantitative descriptors of the oscillations using the ridges of the SST.}}
 Signals of circadian rhythms are slow-varying oscillating signals with a damped amplitude. 
They may reasonably be considered as a monocomponent signal $x(t) = A(t)e^{i\phi(t)}$ with slowly varying amplitude 
\begin{equation*}
	A(t)=\exp(-\lambda t)
\end{equation*} 
with $\lambda >0$ and slowly varying phase $\phi(t)$. 
Hence, the quantitative descriptors of the oscillatory features in the signal corresponds to the time-frequency energy localisation given by the SST \cite{delprat1992asymptotic}. \\
Denote $\phi'(t) = \omega(t)$ the instantaneous frequency, then the curves shaped by the local maxima of the CWT in the time-scale plane follow the instantaneous frequency. 
Hence, the \textbf{quantitative descriptor of the period evolution} is given by 
\begin{equation*}
	\omega_r(b) = \arg\max_\omega |T_x(\omega,b)|,
\end{equation*}
and the instantaneous amplitude is 
$$A_{\rm inst}(b) = |T_x(\omega_r(b), b)|.$$
The parameter of the \textbf{amplitude decay}, $\lambda$, is estimated by fitting the instantaneous amplitude obtained with the CWT with a model $t\mapsto \exp(-\lambda t).$

The continuous wavelet transform and synchrosqueezed transform provide a robust mathematical framework for instantaneous frequency estimation through ridge extraction which ensures reliable performance under appropriate conditions: reasonable signal-to-noise ratio, sufficiently slow rate of frequency variation, the correct wavelet choice and its time-frequency concentration, etc (cf. \cite{lilly2017element}).

Several methods are available to extract the ridges 
\cite{carmona1997characterization, slavivc2003damping, qin2016adaptive} and they mostly involve:  
\begin{itemize}
	\item to compute the CWT over appropriate scales and times
	\item to find the frequency $\omega_r(b)$ that maximises $|T_x(\omega,b)|$ for each time localisation parameter $b$,
	\item to apply smoothing constraints to ensure the ridge continuity,
	\item and to convert the ridge scales to instantaneous frequency.
\end{itemize}
The key strength of the CWT-based approach lies in its adaptive time-frequency resolution, making it particularly suitable for analysing signals with time-varying spectral content. 
Moreover, the features extraction using this method also has the advantage of enabling accurate estimation despite a very low signal-to-noise ratio (cf. Figure \ref{fig:sanity_check}).

	\begin{figure}[h!]
	\centering
	\includegraphics[width=\textwidth]{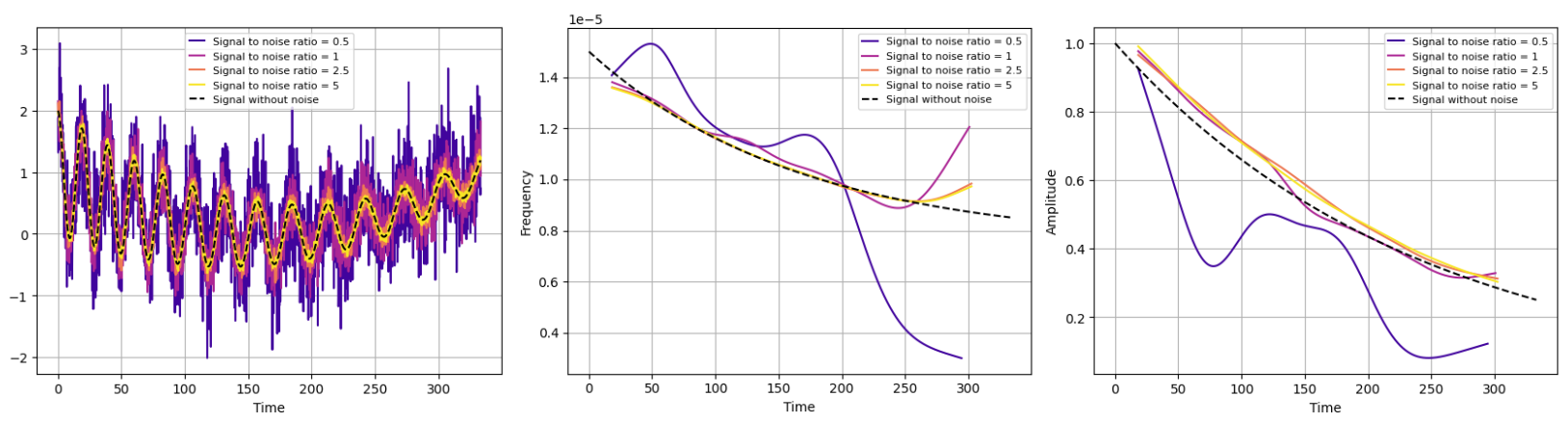}
	\caption{\small
		Robustness of the wavelet-based feature extraction method in the presence of noise. 
		The signal-to-noise ratio modulates the additive noise effect. 
		(Left) Synthetic signal (dashed black) with additive noise at multiple levels (color), illustrating increasing signal degradation. 
		(Middle) Estimated instantaneous frequency for each noise level.
		(Right) Estimated instantaneous amplitude for each noise level.  
		The method reliably extracts the amplitude and period for signal-to-noise ratios up to 0.5, thereby demonstrating its robustness against additive noise.
		}
	\label{fig:sanity_check}
\end{figure}

\section{Modelling assumptions}
\label{sec:ap_para}

In this Appendix, we detail the modelling choice about the parameters values and the numerical methods.
Table \ref{tab:fixed_parameters} recalls the parameters that are fixed \textit{a priori} based on established properties of mammalian circadian clocks and for computational convenience.

\begin{table}[ht]
\centering
\begin{tabular}{|c|cccc|}
\hline
Parameters & $L^{\rm WT}$ & $L^{\rm Cry2KO}$ & $\Omega_i$ & $P^0$ \\
\hline
Values & $24.46~\mathrm{h}$ & $28.42~\mathrm{h}$ & $1$ & Eq.\eqref{eq:P_0} \\
\hline
\end{tabular}
\caption{Fixed parameters for the system \eqref{eq:adaptive_kura}-\eqref{eq:luci}.}
\label{tab:fixed_parameters}
\end{table}

\noindent Parameters dictating the dynamical behaviour of the system are described in Section \ref{sec:param_filter} and the ranges of those parameters are given in Table \ref{tab:para}.
In the following, we give supplementary information for the sake of clarity and reproducibility.
\\

\noindent\textbf{Additional information about the spread of the initial condition $\sigma$:}

\begin{minipage}{0.64\textwidth}
   The synchronisation state of the population at initial time of the experiment is unknown. Then the initial phase distribution is parametrised to account for different levels of synchronisation and we assume it follows a truncated Gaussian distribution on the interval $[0,L]$, centered at $L/2$ with density $f$ such that:
$$
	f(\phi)d\phi=\frac{\exp\left(-\frac{(\phi-L/2)^2}{2\sigma^2}\right)}
	{\int_0^L \exp\left(-\frac{(x-L/2)^2}{2\sigma^2}\right)\,dx} d\phi,
	 \phi\in[0;L],
	$$
	where parameter $\sigma\in[10^{-2},10^{2}]$ controls the spread of the distribution.
\end{minipage}
\hfill
\begin{minipage}{0.35\textwidth}
	\includegraphics[width=\textwidth]{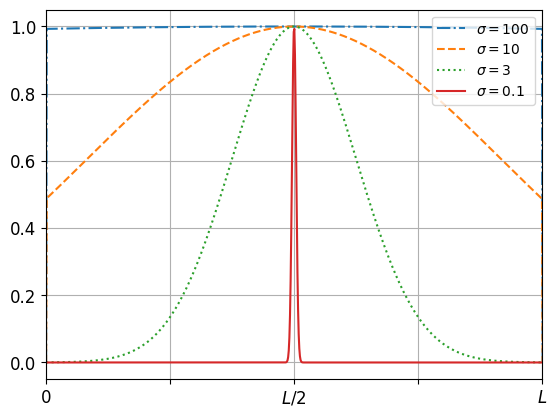}
\end{minipage}
\\

\noindent\textbf{Additional information about the width of the coupling window $\Delta_\theta$:} The coupling window is defined with two additional parameters $\theta_\text{min}$ (resp. $\theta_\text{max}$) which characterised the beginning (resp. the end) of the stage of the circadian cycle where cells are susceptible to intercellular coupling. 
We set $\theta_{\min} = L/2 - \Delta_\theta/2$ and $\theta_{\max} = L/2 + \Delta_\theta/2$, which guarantees that the coupling window $[\theta_{\min},\theta_{\max}]$ is contained within the circadian cycle interval $[0,L]$.
\\

\begin{table}[ht]
\centering
\begin{tabular}{|c|cccc|}
\hline
Parameters & $N$ & $dt$ & $T_{\max}$ & $s$ \\
\hline
Values & 100 & 10 min & $2T_{\mathrm{data}}$ & 20 \\
\hline
\end{tabular}
\caption{Discretisation parameters used for numerical simulations.}
\label{tab:numerical_parameters}
\end{table}

\noindent\textbf{Additional information about the discretisation parameters:} The parameters for the numerical approximations are described in Table \ref{tab:numerical_parameters}. 
The simulations are performed in Python using the solve\_ivp solver with the LSODA method and the following is assumed: 
\begin{itemize}
    \item \textbf{Number of oscillator $N$.}
Since we compare the normalised output $Y$ of equation \eqref{eq:luci} with the normalised data, the number of oscillator does not directly affects the comparison. 
To keep a balance between the biological reality and computational efficiency, we choose to simulate a population of $N = 100$ oscillators.
\item \textbf{Time step of the computation $dt$.} The time step used for the simulations is chosen to match the sampling interval of the experimental data, i.e., 10 minutes, allowing direct comparison between the model output and the data.
\item \textbf{Maximum simulation time $T_{\max}$.} To ensure that the model output can be compared with the experimental data, the maximum simulation time is set to twice the maximum duration of the dataset.
\end{itemize}

\noindent
\begin{minipage}{0.6\textwidth}
\begin{itemize}
    \item \textbf{Phase localised coupling function $\varphi$.}
We define $\varphi$ on a single cycle and extend it by periodicity to obtain an $L$-periodic function. 
Numerically, $\varphi$ is given by
$$\varphi(\theta) = \frac{\tanh(s(\theta - \theta_{\min})) + \tanh(s(\theta_{\max} - \theta))}{2}.$$
\end{itemize}
\end{minipage}
\hfill
\begin{minipage}{0.4\textwidth}
	\begin{tikzpicture}[scale=0.5]
		
		\def\L{24}
		\def\thetamin{8}
		\def\thetamax{16}
		\def\sa{20}
		\def\sb{2}
		
		\begin{axis}[
			width=11cm,
			height=6cm,
			xlabel={$\theta$},
			ylabel={$\varphi(\theta)$},
			domain=0:\L,
			samples=300,
			axis lines=left,
			ymin=0, ymax=1.1,
			xmin=0, xmax=\L,
			xtick={0,\thetamin,\thetamax,\L},
			xticklabels={$0$,$\theta_{\min}$,$\theta_{\max}$,$L$},
			legend style={
				at={(0.98,0.98)},
				anchor=north east,
				draw=none,
				fill=none
			}
			]

			\addplot[blue, thick] 
			{(tanh(\sa*(x-\thetamin)) + tanh(\sa*(\thetamax-x)))/2};
			\addlegendentry{$s=20$}
			
			\addplot[red, thick] 
			{(tanh(\sb*(x-\thetamin)) + tanh(\sb*(\thetamax-x)))/2};
			\addlegendentry{$s=2$}
		\end{axis}
		
	\end{tikzpicture}
\end{minipage}
\begin{itemize}
    \item[ ] As seen in the figure, the parameter $s$ defines the slope of the function: the larger $s$ is, the closer $\varphi$ approaches an indicator function.
We fix $s = 20$ to balance the smoothness of the function with a sufficiently sharp transition at the edges of the coupling interval.
\end{itemize}

\end{document}